\documentclass[11pt]{amsart}
\usepackage{amssymb,amsmath,amsthm,amsfonts,mathrsfs}
\usepackage[hmargin=3cm,vmargin=3.5cm]{geometry}
\usepackage{hyperref}
\usepackage{color}
\usepackage{graphicx,psfrag}
\usepackage{amscd}  
\usepackage[shellescape]{gmp}
\usepackage{stmaryrd}  
\usepackage[all,2cell]{xy} \UseAllTwocells \SilentMatrices
\usepackage{tikz-cd}
\usepackage[dvips]{epsfig}

\newtheorem{theorem}{Theorem}[section]

\newtheorem{prop}[theorem]{Proposition}
\newtheorem{cor}[theorem]{Corollary}

\let\oldtocsection=\tocsection
\let\oldtocsubsection=\tocsubsection
\renewcommand{\tocsection}[2]{\hspace{0em}\oldtocsection{#1}{#2}}
\renewcommand{\tocsubsection}[2]{\hspace{1em}\oldtocsubsection{#1}{#2}}

\usepackage{chngcntr}
\counterwithin{figure}{subsection}

\title[Evaluating thin flat surfaces]{Evaluating thin flat surfaces}
\author{Mikhail Khovanov}
 \address{Department of Mathematics, Columbia University, New York, NY 10027, USA}
 \email{\href{mailto:khovanov@math.columbia.edu}{khovanov@math.columbia.edu}}
 \author{You Qi}
 \address{Department of Mathematics, University of Virginia, Charlottesville, VA 22904, USA}
 \email{\href{mailto:yq2dw@virginia.edu}{yq2dw@virginia.edu}}
 \author{Lev  Rozansky}
 \address{Department of Mathematics, University of North Carolina, Chapel Hill, NC 27599, USA}
 \email{\href{rozansky@math.unc.edu}{rozansky@math.unc.edu}}
\date{September 4, 2020}

\begin{document}

\def\R{\mathbb R}
\def\Q{\mathbb Q}
\def\Z{\mathbb Z}
\def\N{\mathbb N} 
\def\C{\mathbb C}
\def\S{\mathbb S}
\def\CP{\mathbb P}
\renewcommand\SS{\ensuremath{\mathbb{S}}}

\def\l{\lbrace}
\def\r{\rbrace}
\def\o{\otimes}
\def\lra{\longrightarrow}
\def\Ext{\mathrm{Ext}}
\def\mc{\mathcal}
\def\mf{\mathfrak} 
\def\uFr{\underline{\mathrm{Fr}}}
\def\mcC{\mathcal{C}}
\def\undb{\underline{b}}
\def\undc{\underline{c}}

\def\mfgl{\mathfrak{gl}}
\def\mfglN{\mathfrak{gl}_N} 
\def\Cat{\mathrm{Cat}}
\def\ovF{\overline{F}}
\def\ovb{\overline{b}}
\def\tr{{\sf tr}} 
\def\det{{\sf det }} 
\def\tral{\tr_{\alpha}}
\def\one{\mathbf{1}}   
\def\tZ{\widetilde{Z}}  
\def\wtP{\widetilde{P}}
\def\wtQ{\widetilde{Q}}
\def\ha{\widehat{\alpha}}

\def\lra{\longrightarrow}
\def\kk{\mathbf{k}}  

\def\gdim{\mathrm{gdim}}  
\def\rk{\mathrm{rk}}

\def\tfs{\mathrm{TFS}} 
\def\ktfs{\kk\tfs}  
\def\vtfs{\mathrm{VTFS}}  
\def\stfs{\mathrm{STFS}} 
\def\gtfs{\mathrm{TFS}} 
\def\dtfs{\mathrm{DTFS}} 
\def\udtfs{\underline{\dtfs}}  
\def\SS{\mathcal{C}} 
\def\vSS{\mathcal{V}\SS}    
\def\dSS{\mathcal{D}\SS}    
\def\sSS{\mathcal{S}\SS}  
\def\udSS{\underline{\dSS}}  

\def\tfsr{\tfs^{(r)}} 
\def\vtfsr{\vtfs^{(r)}}
\def\stfsr{\mathrm{STFS}^{(r)}}
\def\dtfsr{\mathrm{D}\tfsr}  
\def\udtfsr{\underline{\mathrm{DTFS}}^{(r)}_{\,\alpha}}

\def\undern{\mathbf{n}} 
\def\underm{\mathbf{m}}

\def\mcT{\mathcal{T}} 
\def\mcTcheck{\mcT^{\vee}} 
\def\Rec{\mathrm{Rec}} 

\def\Pa{\mathrm{Pa}}   
\def\Cob{\mathrm{Cob}} 
\def\Cobtwo{\Cob_2}   
\def\Kob{\mathrm{Kob}}
\def\PCobn{\mathrm{PCob}}  
\def\Cobal{\Cob_{\alpha}}  
\def\Cobalp{\Cob_{\alpha}'} 
\def\Kobal{\Kob_{\alpha}}   
\def\PCob{\mathrm{PCob}}
\def\PCobal{\PCob_{\alpha}}
\def\Kar{\mathrm{Kar}}   

\def\dmod{\mathrm{-mod}}   
\def\pmod{\mathrm{-pmod}}    

\newcommand{\brak}[1]{\ensuremath{\left\langle #1\right\rangle}}
\newcommand{\oplusop}[1]{{\mathop{\oplus}\limits_{#1}}}
\newcommand{\ang}[1]{\langle #1 \rangle } 
\newcommand{\bbn}[1]{\mathbb{B}^{#1}}

\newcommand{\mcA}{{\mathcal A}}
\newcommand{\cZ}{{\mathcal Z}}
\newcommand{\sq}{$\square$}
\newcommand{\bi}{\bar \imath}
\newcommand{\bj}{\bar \jmath}

\newcommand{\undn}{\mathbf{n}}
\newcommand{\undm}{\mathbf{m}}

\newcommand{\Hom}{\mbox{Hom}}
\newcommand{\Ind}{\mbox{Ind}}
\newcommand{\id}{\mbox{id}}
\newcommand{\Id}{\mbox{Id}}
\newcommand{\End}{\mathrm{End}}
\newcommand{\iHom}{\underline{\mbox{Hom}}}

\begin{abstract}
We consider recognizable evaluations for a suitable category of oriented two-dimensional cobordisms with corners between finite unions  of  intervals. We call such cobordisms  thin  flat  surfaces. An evaluation is  given by a power series  in two  variables. Recognizable evaluations correspond  to  series that are ratios of a  two-variable polynomial by the product of two  one-variable polynomials, one for each variable. They are also in a bijection with isomorphism  classes of commutative Frobenius algebras on two generators with a nondegenerate trace fixed. The latter algebras of  dimension n correspond to  points on the  dual tautological bundle on the Hilbert scheme of n points on the affine  plane, with  a certain divisor removed  from the bundle. A  recognizable evaluation gives  rise  to a functor  from the above cobordism  category of  thin  flat surfaces to the category of finite-dimensional  vector spaces. These functors may  be non-monoidal  in interesting cases. To a  recognizable evaluation we  also assign an analogue of the Deligne category and of its quotient by the ideal of  negligible morphisms. 
\end{abstract}

\maketitle
\tableofcontents

%
%

\section{Introduction}
Universal constructions of topological theories~\cite{BHMV,Kh1,RW} that are not necessarily multiplicative~\cite{FKNSWW} are interesting even in dimension two~\cite{Kh2,KS},  providing examples somewhat different from  commutative Frobenius algebras for the invariants of two-dimensional cobordisms. In this note we consider  the analogue of  the  latter construction for oriented  two-dimensional  cobordisms with corners. For simplicity we restrict to cobordisms between finite unions of  intervals;  boundary points of the intervals give rise to corners of   cobordisms. Furthermore, we require  that each connected component of a cobordism has non-empty  boundary, which is a natural  condition  when excluding cobordisms with corners that  have circles as some boundary components. 

Cobordisms that we consider can  be ``thinned" to  consist of ribbons glued  to  disks and can be depicted in the plane  as regular neighbourhoods of immersed  graphs, see Figure~\ref{fig1_2} below for an  example. For this  reason we refer to these  cobordisms as \emph{thin flat  surfaces} or \emph{tf-surfaces} throughout the paper. When viewed as a   morphism in the appropriate category $\tfs$ of thin flat cobordisms, a particular immersion of  the surface into  the plane  is  inessential,  and  morphisms are  equivalence classes of such  cobordisms modulo diffeomorphisms that fix the  boundary. 

\vspace{0.07in}

The category $\tfs$ admits an analogue of $\alpha$-evaluations from~\cite{Kh2,KS,KKO}. This  time closed connected morphisms $S$ (connected endomorphisms of the  unit object $0$, the empty union of  intervals) are  parametrized by two non-negative integers $(\ell,g)$, where $\ell+1$ is  the number of boundary  components of $S$  and $g$  is the genus. Assigning an  element $\alpha_{\ell,g}$ of  the ground field $\kk$  (or a ground commutative  ring $R$) to such a component and extending multiplicatively to disjoint unions gives an evaluation $\alpha$ on endomorphisms of the unit object $0$. Evaluation $\alpha$ can be conveniently encoded as power series 
\begin{equation}\label{eq_Z_rat_2}
    Z_{\alpha}(T_1,T_2) = \sum_{k,g\ge 0}\alpha_{k,g} T_1^k T_2^g, \ \ 
     \alpha=(\alpha_{k,g})_{k,g\in \Z_+}, \ \ \alpha_{k,g} \in \kk , 
\end{equation} 
where the degree of the first  variable $T_1$ counts  ``holes" in a  cobordism (a disk has no holes  and  an annulus has one hole) and $T_2$ keeps track of the genus.  

\vspace{0.1in} 

With $\alpha$ as above and $n\ge 0$, one can  define a bilinear form on 
a $\kk$-vector  space with a basis given by equivalence  classes of thin  flat  surfaces with $n$  boundary intervals.  The quotient by the kernel of the bilinear form is a vector space  $A_{\alpha}(n)$. 
The collection  of  quotient spaces $\{A_{\alpha}(n)\}_{n\ge 0}$ is what we refer to as \emph{the universal construction} for the category $\tfs$, given $\alpha$.

\vspace{0.1in} 

The spaces $A_{\alpha}(n)$ rarely satisfy the Atiyah factorization axiom, that is, the relation 
\begin{equation*}
A_{\alpha}(m+n)\cong A_{\alpha}(m)\otimes A_{\alpha}(n)
\end{equation*}  
does not  hold. From the quantum field theory (QFT) perspective, this violation may happen if the 2-dimensional QFT is embedded as a 2-dimensional defect inside a higher-dimensional QFT. 

\vspace{0.1in} 

It  is straightforward to see that $A_{\alpha}(n)$  is finite-dimensional for all  $n$ iff $A_{\alpha}(1)$ is finite-dimensional iff the series  (\ref{eq_Z_rat_2}) is \emph{recognizable} or \emph{rational} (terms from the control theory and the theory of noncommutative  power series). Recognizable power series in this case have the  form 
\begin{equation}\label{eq_recog_2}
    Z_{\alpha}(T_1,T_2) = \frac{P(T_1,T_2)}{Q_1(T_1)Q_2(T_2)},
\end{equation}
that is, $Z_{\alpha}$ can be  written  as a ratio of  a polynomial in $T_1,T_2$ and  two one-variable polynomials,  see Proposition~\ref{prop_recog_alpha} and Fliess~\cite{F}. 

Constructions of~\cite{KS} go through for the category of thin flat surfaces and any recognizable $\alpha$ as above. We define the category $\stfs_{\alpha}$ (skein thin flat surfaces) where homs are finite linear  combinations of cobordisms, closed components evaluate to coefficients of $\alpha$, and there are skein relations given by adding holes and handles to a component of a cobordism and equating to zero linear combinations corresponding to elements  of  the kernel ideal $I_{\alpha}\subset \kk[T_1,T_2]$ associated to $\alpha$ and also known  as \emph{the  syntactic ideal} of rational series $\alpha$. A  two variable polynomial $z=z(T_1,T_2)$ is in $I_{\alpha}$ iff $\alpha(zf)=0$ for any polynomial $f\in \kk[T_1,T_2]$, with $\alpha(T_1^{\ell}T_2^g)=\alpha_{\ell,g}$ extended to a linear map  $\kk[T_1,T_2]\stackrel{\alpha}{\lra}\kk$.

For the rest of the paper  we change our  terminology and call connected components of a  thin  flat surface that  have neither top nor  bottom  boundary intervals \emph{floating} components instead of \emph{closed} components, since they otherwise have  boundary, what we call \emph{side} boundary, that  is present inside the  cobordism but not at its top or bottom. This avoid possible confusion with the  usual  notion  of a closed surface. A non-empty thin flat  surface  is never closed in the latter sense.  

\vspace{0.07in}

The category  $\stfs_{\alpha}$ has finite-dimensional hom  spaces.  Taking the additive Karoubi envelope
of this category to form 
\begin{equation*}
    \dtfs_{\alpha} \ := \ \Kar(\stfs^{\oplus}_{\alpha})
\end{equation*}
gives an idempotent-complete $\kk$-linear rigid symmetric monoidal category $\dtfs_{\alpha}$ which is the analogue of the Deligne category~\cite{D,CO,EGNO} for $\tfs$ and recognizable series  $\alpha$  in two variables. 

\vspace{0.07in}

Once we pass to $\kk$-linear combinations of cobordisms,  and $\alpha$ is available to evaluate floating  cobordisms, there is a trace map on endomorphisms of any  object $n$. It  is given by closing each term in the linear  combination of tf-surfaces  describing the endomorphism via $n$ strips into a floating tf-surface and evaluating it via $\alpha$. Consequently, one can form the ideal  $J_{\alpha}$ of negligible morphisms~\cite{D,CO,EGNO,KS} and quotient the category by that ideal. 

We call the quotient category \emph{gligible quotient} to avoid the awkward-sounding word ``non-negligible quotient"  and mirroring the terminology from~\cite{KKO}. The gligible quotient $\gtfs_{\alpha}$ of the skein category  $\stfs_{\alpha}$ carries non-degenerate bilinear forms on its hom spaces and otherwise shares key properties of $\stfs_{\alpha}$: objects are  non-negative integers, category $\gtfs_{\alpha}$ is rigid symmetric  tensor, and the hom spaces  are finite-dimensional over $\kk$. 

Likewise, the Deligne category $\dtfs_{\alpha}$  has the gligible  quotient  $\udtfs_{\,\alpha}$ by the ideal of negligible  morphisms. 
The same category can be recovered as the  additive Karoubi  closure of  $\gtfs_{\alpha}$. 

Section~\ref{subsec_summary} and diagram (\ref{eq_seq_cd}) contain a summary of  these categories and key functors  relating them.

\vspace{0.07in}

Similar to~\cite{EO,EGNO,KKO}, it is natural to ask under what conditions will $\udtfs_{\,\alpha}$ be semisimple. Unlike~\cite{KKO,Kh2}, we do not work out any specific examples of these categories here and leave that to an interested reader or another time.

\vspace{0.07in}

Our evaluation ${\alpha}$ is encoded by a power series $Z_{\alpha}$ in two  variables  (\ref{eq_Z_rat_2}), and the recognizable series $Z_{\alpha}$  gives rise to a finite-codimension ideal $I_{\alpha}$ in $\kk[T_1,T_2]$, the largest ideal contained in the  hyperplane $\ker(\alpha)$. Such an ideal defines a point  on  the Hilbert scheme of the affine plane $\mathbb{A}^2$. We discuss the relation to the Hilbert  schemes in Section~\ref{sec_sample} and explain a bijection between recognizable power series with  the ideal $I_{\alpha}$ of codimension $k$ and points in the complement $\mcTcheck_k\setminus D_k$ of the dual tautological bundle $\mcTcheck_k$ on the Hilbert scheme and  a  suitable divisor $D_k$ on it. 

It is not  clear whether the appearance of the Hilbert scheme of  $\mathbb{A}^2$ is a bug or a feature. In Section~\ref{sec_modif} we explain two generalizations of our construction.  One of them involves ``coloring" side  boundary components of a thin flat surface into  $r$ colors. For the resulting  category,  recognizable series depends  on $r+1$ parameters (generalizing from $2$ parameters for $r=1$), and one would  get a generalization of our construction from the Hilbert scheme of  $\mathbb{A}^2$ to that of $\mathbb{A}^{r+1}$, with the appropriate divisor removed from the dual tautological bundle in both cases. Of course, the Hilbert scheme has vastly different properties and  uses in the case of  algebraic surfaces versus higher dimensional varieties.

\vspace{0.07in}

The other  generalization considered in that section is given by extending the $\tfs$ (thin flat surfaces) category by allowing  closed components and circles  as boundaries. This  corresponds to the usual  category 
of two-dimensional oriented cobordisms with boundary and  corners studied in~\cite{MS,LP,C,SP} and  other  papers. Objects  of that category are finite disjoint unions of intervals and circles. 
We  briefly touch on this generalization  and  explain encoding of recognizable series  via certain  rational power series in this case  as well. 

\vspace{0.07in}

Relations between Frobenius algebras, recognizable  power series, codes and two-dimensional TFTs are considered in Friedrich~\cite{Fr}, which is quite close in  spirit  to  this paper. 

A possible relation between moduli spaces of $SU(m)$  instantons on $\R^4$ (the Hilbert scheme of $\C^2$ corresponds to $U(1)$ case) and control theory is explored  in~\cite{H,S} and the  follow-up papers. We do not know how  to connect it to the constructions in the present paper. 

\vspace{0.1in} 

{\bf Acknowledgments.} M.K. was partially supported by the NSF grant  DMS-1807425 while working on this paper. Y. Q. was partially supported by the NSF grant DMS-1947532.
L.R. was partially supported by the NSF grant DMS-1760578.

%
%

\section{The category of thin flat surfaces}


\subsection{Category \texorpdfstring{$\tfs$}{TFS}} 

We introduce the category $\tfs$ of \emph{thin  flat surfaces}. Its objects are non-negative integers $n\in \Z_+=\{0,1,2,\dots\}$. An object is represented by $n$  intervals $I_1,\dots, I_n$ placed along the  $x$-axis in the  $xy$-plane. A morphism from $n$ to  $m$ is a ``thin'' surface   $S$ immersed  in $\R\times [0,1]$ connecting $n$ intervals on the line $\R^2\times \{0\}$ with  $m$ intervals on the line $\R^2\times \{1\}$. The immersion  map $S\lra \R^2\times [0,1]$ is a local diffeomorphism, but the image of $S$ may have overlaps, that can be thought of as virtual overlaps and ignored. The surface $S$ inherits an orientation from its immersion into $\R\times [0.1]$. Restricting  to the complement of the boundary of $S$, the immersion is open. 

Alternatively, the immersion can be perturbed to an embedding of $S$ into $\R^2\times [0,1]$ by turning overlaps into over- and under-crossings of strips of a surface. This can be done just for  aesthetic  purposes, and whether one chooses  an over- or  an under-crossing  does not matter  for the morphism associated to the  surface. 

The boundary of  $S$ consists of  several circles  (at  least one circle unless $S$ is the  empty surface) and decomposes  into $n+m$ disjoint intervals that  constitute \emph{horizontal} boundary $\partial_h S$ and $n+m$ intervals and  some  number of circles that constitute \emph{side}, or \emph{vertical}, or \emph{inner} boundary $\partial_v S$ of $S$:
\begin{equation*}
    \partial S = \partial_h S \cup \partial_v  S.
\end{equation*}
Horizontal intervals that  constitute $\partial_h S$  are the intersections of $S$ with $\R\times\{0,1\} \subset\R\times [0,1]$. Vertical boundary $\partial_v S$  is  the  closure of  the intersection of  $\partial S$ with $\R\times (0,1)$. The intersection  $\partial_h S \cap \partial_v S$ consists of $2(n+m)$ boundary points of  the horizontal intervals. These are also the \emph{corners} of the  surface $S$. 

In the graphical depicitions of thin flat surfaces below, we will draw horizontal boundary segments as red intervals, and vertical boundary components as green arcs for better visualization (Figure~\ref{fig1_1} and Figure~\ref{fig1_2} right),  but the figures can also be  viewed and carry full  information  in greyscale. Starting from Figure~\ref{fig1_2}, we  depict tf-surfaces in light aquamarine.  

\begin{figure}[h]
\psfrag{1}{$1$}
\psfrag{2}{$2$}
\psfrag{3}{$3$}
\psfrag{m}{$m$}
\psfrag{n}{$n$}
\begin{center}
\includegraphics[scale=0.35]{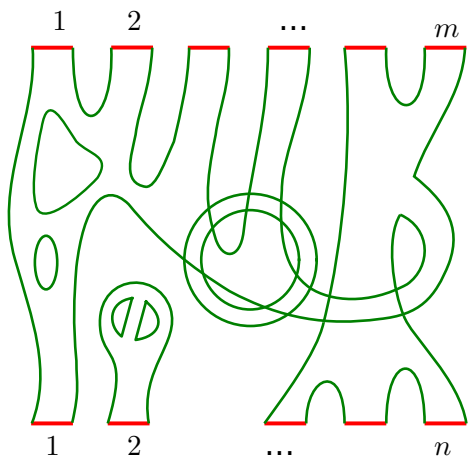}
\caption{A thin  flat surface  in $\R\times[0,1]$.}
\label{fig1_1}
\end{center}
\end{figure}

Another way to describe $S$  is to  immerse a  finite unoriented graph $\Gamma$, possibly with multiple edges and loops, into  the strip $\R\times [0,1]$, via the immersion $\jmath: \Gamma \lra \R\times [0,1]$. The graph $\Gamma$ may have several boundary vertices $v$ of valency $1$ such  that  $\jmath(v)\in \R\times\{0,1\}   $.  The  remaining  vertices are mapped inside the strip. The immersion is disjoint on vertices. Edges of $\Gamma$ may intersect in $\R\times [0,1]$. We consider these virtual intersections and not vertices. An example of $\Gamma$ and $\jmath$ is  shown in Figure~\ref{fig1_2} left. 

\begin{figure}[h]
\psfrag{1}{$1$}
\psfrag{2}{$2$}
\psfrag{3}{$3$}
\psfrag{m}{$m$}
\psfrag{n}{$n$}
\begin{center}
\includegraphics[scale=0.70]{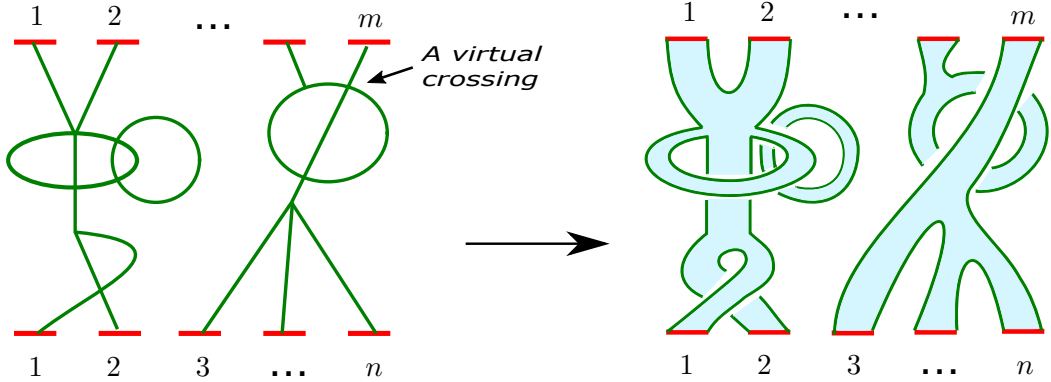}
\caption{An immersed graph $\Gamma$ in $\R\times [0,1]$ and associated thin flat surface $N(\Gamma,\jmath)$.}
\label{fig1_2}
\end{center}
\end{figure}

Taking a regular neighbourhood $N(\Gamma,\jmath)$ of  $\Gamma$  under $\jmath$, locally  in $\Gamma$,  results in a thin flat surface  $N(\Gamma)$.  Vice versa, any thin flat surface  $S$  can be deformed  to  the surface $N(\Gamma)$  for some $\Gamma$.

Take a thin flat surface $S$ and forget the embedding into $\R\times [0,1]$, only remembering boundary intervals and their order, on both top and bottom lines. In this way we view $S$ as a cobordism between ordered collections of oriented intervals (induced by the  orientation  of $\R$, say from left to right). The cobordism $S$ has corners (unless $n=m=0$) and two types  of boundary, as discussed. By definition, two cobordisms $S_1,S_2$ represent the same morphism if  they are diffeomorphic rel horizontal boundary, that is,  keeping all horizontal  boundary points fixed. 

The category $\tfs$ is symmetric monoidal, and a  possible set of  generating  morphisms is shown  in Figure~\ref{fig1_5}.  We have included the  identity morphism $\id_1$ into the Figure to emphasize that the  identity morphism $\id_n$ is represented by the surface which is the direct product of the   disjoint union of $n$ intervals (representing object $n$) and $[0,1]$. The permutation morphism $P$ of $2=1\otimes  1$ , shown on the right, is part of the symmetric  monoidal structure on  $\tfs$ and squares  to the  identity.

\begin{figure}[h]
\begin{center}
\psfrag{i}{$\iota$}
\psfrag{E}{$\epsilon$}
\psfrag{M}{$m$}
\psfrag{D}{$\Delta$}
\psfrag{Id}{$\mathrm{id}_1$}
\psfrag{P}{$P$}
\includegraphics[scale=0.75]{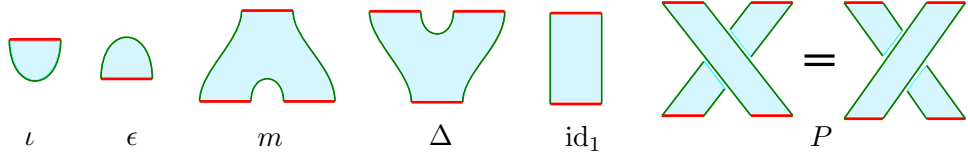}
\caption{A set  of  generating morphisms.  From left to right: $\iota, \epsilon, m, \Delta$ are morphisms from $0$ to $1$, from $1$ to $0$, from $2$ to  $1$ and  from $1$ to $2$, respectively. The rightmost morphism $P$ is the permutation morphism on  $1\otimes 1= 2$. Identity morphism $\id_1$ of object $1$ is  shown for completeness.}
\label{fig1_3}
\end{center}
\end{figure}

The elements $\iota,\epsilon,m,\Delta,P$
constitute a set  of monoidal generators  of $\tfs$.  Together with the  identity morphism  $\id_1$ they can be used to build any morphism in $\tfs$, via  horizontal and vertical  compositions. In  particular, from these generators we can build the  self-duality morphisms for the object $1$, see Figure~\ref{fig1_4}. 

\begin{figure}[h]
\begin{center}
\includegraphics[scale=0.4]{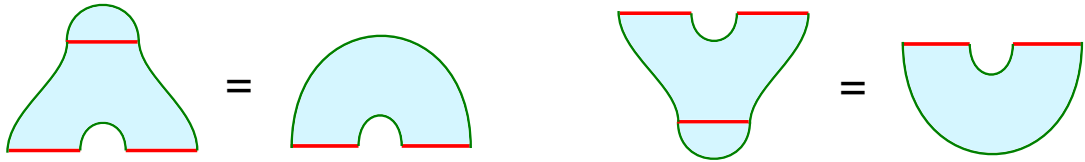}
\caption{Self-duality morphisms $\epsilon\,  m:1\otimes  1 \lra 0$ and  $\Delta\,\iota: 0 \lra 1\otimes 1$ for the  object $1$.}
\label{fig1_4}
\end{center}
\end{figure}

Some relations in the  category $\tfs$ are shown in Figure~\ref{fig1_5}. 

\begin{figure}[h]
\begin{center}
\includegraphics[scale=0.7]{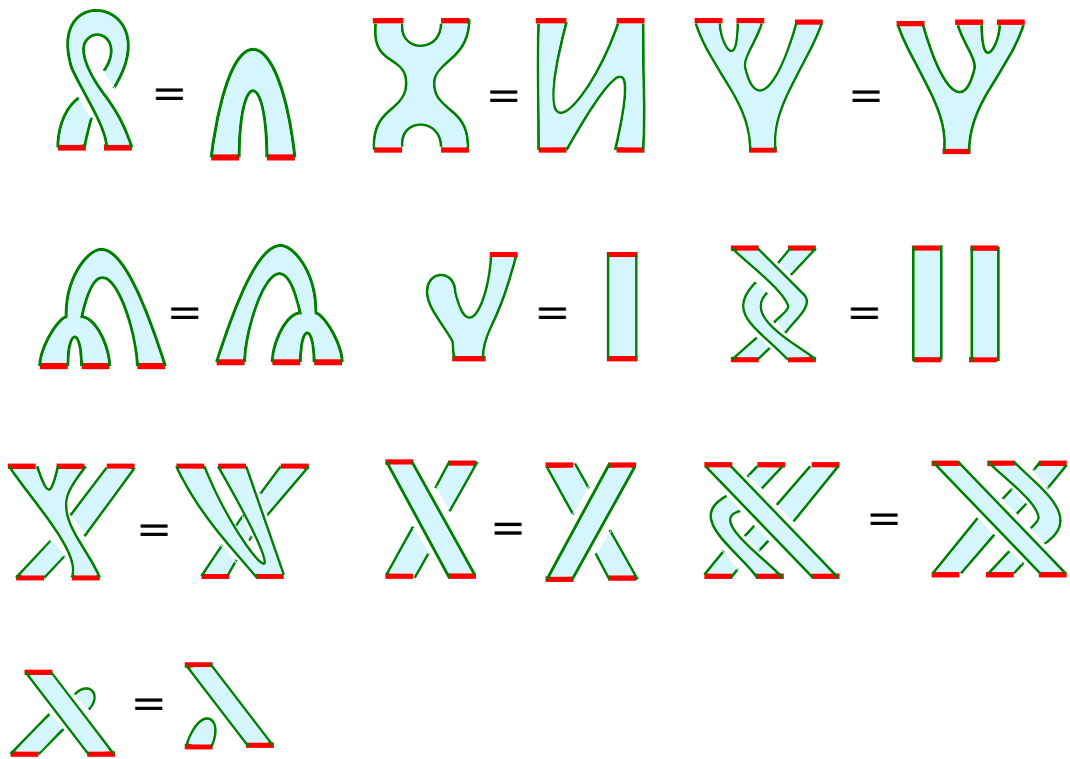}
\caption{Some relations  in $\tfs$.}
\label{fig1_5}
\end{center}
\end{figure}

We call a surface $S$ representing a morphism  from $n$  to $m$  in $\tfs$ a  \emph{thin flat cobordism}   from $n$ to $m$. A thin flat  cobordism $S$  is  a  disjoint  union of  its connected  components $S_1,  \dots,S_k$. Consider one such component $S'$.  It necessarily has non-empty boundary, and  we  can assign to $S'$ non-negative integers $\ell,g$, where $\ell+ 1$ is the number of boundary components and  $g\ge  0$ is the  genus of $S'$. The surface  $S'$ carries an orientation, inherited  via an  immersion from the orientation of the plane. 

We will also call a thin flat surface  a  \emph{tf-surface} and, when viewed as a cobordism, a \emph{tf-cobordism}. 

\vspace{0.1in} 

The morphisms 
\begin{equation*}
  \iota: 0 \lra  1, \ \epsilon: 1\lra 0, \ 
  m : 1\otimes 1 \lra 1, \ \Delta: 1 \lra 1\otimes 1
\end{equation*}
and relations on them show that the object $1$ is a symmetric Frobenius  algebra object in $\tfs$ (top left relation in Figure~\ref{fig1_5} shows that the trace map is  symmetric). It's not a commutative algebra object, since the two morphisms in Figure~\ref{fig1_5_1} left are not equal in $\tfs$.

\begin{figure}[h]
\begin{center}
\includegraphics[scale=0.7]{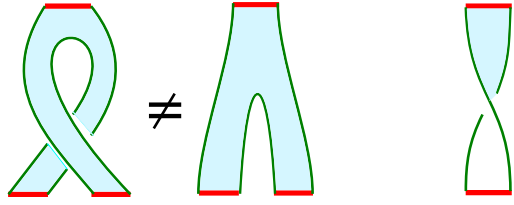}
\caption{Left: object $1$ is not commutative Frobenius. That  the two diagrams on the left  are  not diffeomorphic rel horizontal boundary can be seen easily by examining the matchings on the six corner points in  each diagram provided  by side boundaries. The two matchings of the six points are different, a sufficient condition for the two cobordisms not to be diffeomorphic rel boundary. Right: a diagram that's not a morphism in $\tfs$.}
\label{fig1_5_1}
\end{center}
\end{figure}


\subsection{Classification of thin flat surfaces}

By a \emph{closed} or \emph{floating} tf-surface  $S$ we mean one without horizontal boundary. A floating tf-surface necessarily has side boundary, unless it is the empty surface. Diffeomorphism equivalence  classes of floating   tf-surfaces are in a bijection with  endomorphisms of the object  $0$  of  $\tfs$. Such a surface is a disjoint  union of  its  connected  components, and a component is uniquely determined by its pair $(\ell+1,g)$, $\ell,g\in \Z_+$, the number of boundary components and the genus, respectively.  Any  such pair is realized  by some surface,  since pairs $(0+1,0)$, $(1+1,0)$, and  $(0+1,1)$ are  realized by a disk, an annulus, a flat  punctured torus, see  Figure~\ref{fig1_6}, and  taking band-connected  sum of  surfaces  with invariants  $(\ell_1+1,g_1),$ $(\ell_2+1,g_2)$ yields a  surface with the invariant $(\ell_1+\ell_2+1,g_1+g_2)$.  
Choose a  closed connected tf-surface $S_{\ell+1,g}$, one for each value of these parameters.

\begin{figure}[h]
\begin{center}
\includegraphics[scale=0.5]{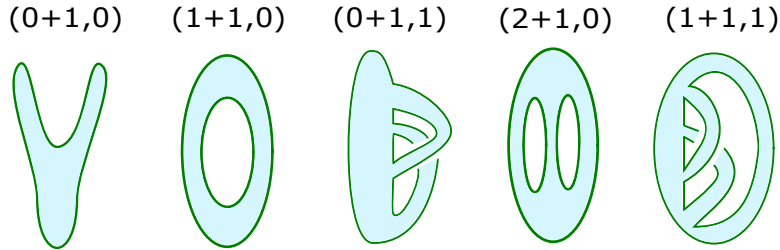}
\caption{Examples of closed  connected tf-surfaces $S_{\ell+1,g}$ for small values of  $\ell$ and  $g$. We explicitly write $\ell+1$ to remember that a surface always have at  least one  boundary component.} 
\label{fig1_6}
\end{center}
\end{figure}

Connected morphisms from  $0$ to $0$ in $\tfs$ are in a bijection with $S_{\ell+1,g}$  as above. Endomorphisms of $0$ in  $\tfs$ is  a free commutative monoid  on  generators $S_{\ell+1,g}$, over all $\ell,g\in \Z_+$, 
\begin{equation*}
    \End_{\tfs}(0) \cong  \langle S_{\ell+1,g}\rangle_{\ell,g\ge 0}. 
\end{equation*}
An element $a \in \End_{\tfs}(0)$ has a unique presentation as a finite product of $S_{\ell+1,g}$'s with positive integer multiplicities, 
\begin{equation*} 
 a  = \prod_{i=1}^k S_{\ell_i+1,g_i}^{r_i}, \ r_i\in \{1,2,\dots\}. 
\end{equation*}
Consider a tf-surface $S$ describing a morphism  from $n$ to $m$ in $\tfs$. It may have some \emph{floating} connected components, that is, those that are disjoint from the  horizontal boundary of $S$.  Each of these components is  homeomorphic  to  $S_{\ell+1,g}$ as above for a unique $\ell,g$. Components of $S$ that have non-empty horizontal boundary are called \emph{viewable} or \emph{visible} components. Any component of $S$ is either \emph{floating} or \emph{viewable}.  We call $S$  \emph{viewable} if it  has no floating components. The empty cobordism is viewable. 

\vspace{0.1in}

The commutative monoid $\End_{\tfs}(0)$ acts on the  set $\Hom_{\tfs}(n,m)$ by taking a cobordism  to its disjoint union with  a  floating cobordism. 
Any morphism $S\in \Hom_{\tfs}(n,m)$ has a unique presentation $S=S_0 \cdot S_1$ where $S_0\in \End_{\tfs}(0)$, $S_1$ is a viewable cobordism in $\Hom_{\tfs}(n,m)$ and dot $\cdot$ denotes the monoid action. In particular, $\Hom_{\tfs}(n,m)$ is a free $\End_{\tfs}(0)$-set with a  ``basis'' of viewable cobordisms. 

\vspace{0.1in} 

Let us specialize to viewable cobordisms $S$. All connected  components of $S$ are viewable and determine a set-theoretic partition of $n+m$ horizontal boundary intervals of $S$. Let us label these boundary intervals from left to right by $1,2,\dots, n$ for the bottom  intervals and $1',\dots, m'$ for the top intervals. 

Each viewable component contains a non-empty subset of this  set of intervals and together  viewable components give a decomposition $\lambda$ of this set  into disjoint sets. We denote  by $D^m_n$  the set of partitions  of these  $n+m$ intervals, so that $\lambda\in D^m_n$. 
To further understand the structure of morphisms, we restrict to the case of connected $S$, thus  
a surface  with one viewable connected component. All horizontal  intervals are in $S$. 

\vspace{0.1in} 

The surface $S$ and its horizontal  boundary  segments inherit orientation from $\R\times [0,1]$ and from induced orientations of the  top and bottom boundary of   $\R\times [0,1]$, 
see  Figure~\ref{fig1_7}. 

\begin{figure}[h]
\begin{center}
\psfrag{1}{$1$}
\psfrag{2}{$2$}
\psfrag{n}{$n$}
\psfrag{1'}{$1^\prime$}
\psfrag{m'}{$m^\prime$}
\psfrag{d1}{$\partial_1$}
\psfrag{d0}{$-\partial_0$}
\psfrag{R01}{$\R\times [0,1]$}
\psfrag{S}{$S$}
\includegraphics[scale=0.75]{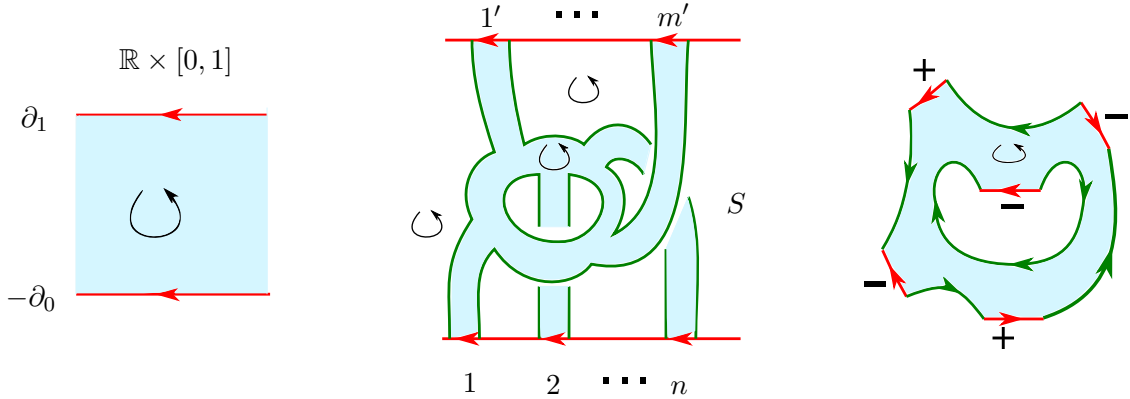}
\caption{Orientation convention for $\R\times[0,1]$, its  top and bottom boundary, surface $S$ and  its  horizontal  and  side boundary.}
\label{fig1_7}
\end{center}
\end{figure}

We use the convention of reversing  the orientation on the source (bottom)  part  of the boundary of a cobordism, see Figure~\ref{fig1_7}. Consequently, bottom intervals $I_1,\dots, I_n$  in $\partial  S$ are oppositely oriented  from the rest of the  boundary, while top intervals $I_{1'},\dots, I_{m'}$ are oriented compatibly with the side boundary orientations, inherited  from that of $S$ and in turn inherited from the orientation of $\R\times [0,1]$. In  Figure~\ref{fig1_7} right we shrank ``tentacles'' of  $S$  into the ``core'' of  $S$ to  make it easier  to see compatible and reverse orientations of the horizontal boundary segments of  $S$. 

\vspace{0.1in} 

We can now classify isomorphism classes of connected tf-cobordisms $S$ from $n$ to  $m$. Such a cobordism has $\ell+1$ boundary circles  and genus  $g$. On  $\ell+1$ boundary circles choose $n+m$ non-overlapping intervals and label them $1,\dots, n, 1',\dots, m'$. Choose an orientation of the interval $1'$ or, if $m=0$, orientation of interval  $1$. 

The orientation of the interval $1'$  induces an orientation of that  boundary component of $S$  and hence of $S$ itself. One then  gets induced orientations for all boundary components  of $S$. Horizontal parts  of $\partial  S$ for the intervals $2',\dots,m'$  are then oriented  compatibly  with the boundary, while  those corresponding  to the intervals  $1,\dots, n$ in  the opposite way from that for the boundary. 

Horizontal intervals  on the $\ell+1$ boundary components determine a partition  of 
\begin{equation*}
\N^m_n \ := \ \{1,\dots,n,1',\dots ,m'\}
\end{equation*} into $\ell+1$ disjoint  subsets, possibly with some  subsets  empty. Orientations of boundary components induce a cyclic order on elements of  each subset, where one goes along a component in the  direction  of its orientation and records horizontal intervals that  one encounters. We call an instance of this data a locally cyclic partition of $\N^m_n$ together with a choice of genus $g\ge 0$. Denote the  set of locally cyclic partitions of $\N^m_n$ by $D^m_{n,cyc}$ and by $D^m_{n,cyc}(\ell)$ if the number  of components is fixed to be $\ell+1$. This  time, empty components are allowed.  They correspond to components of  $\partial S$ disjoint from the  boundary $\R\times \{0,1\}$ of the strip. We have
\begin{equation*}
    D^m_{n,cyc} = \bigsqcup_{\ell\ge 0} d^m_{n,cyc}(\ell). 
\end{equation*}

For the example in Figure~\ref{fig1_7} we have  $n=3$, $m=2$, the  set of horizontal intervals is $\{1,2,3,1',2'\}$, there  are  two components  ($\ell=1$), and the  cyclic orders  are 
$(1',1,2',3)$  and  $(2)$. 

\begin{figure}[h]
\begin{center}
\psfrag{1}{$1$}
\psfrag{2}{$2$}
\psfrag{3}{$3$}
\psfrag{4}{$4$}
\psfrag{2'}{$2^\prime$}
\psfrag{3'}{$3^\prime$}
\psfrag{1'}{$1^\prime$}
\includegraphics[scale=0.7]{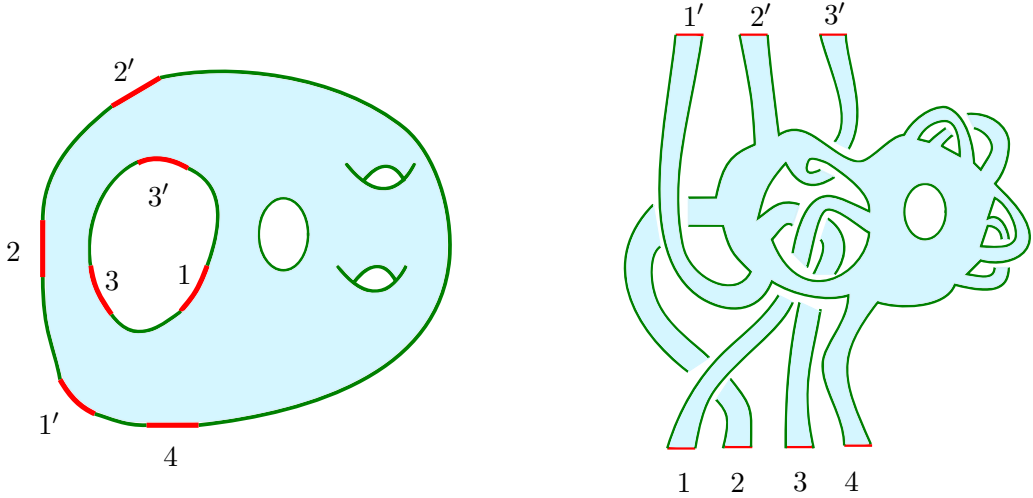}
\caption{ Left:  converting the partition and genus data into a surface with  boundary and labelled  edges on  the boundary. One boundary component (inner right)  does  not carry labelled edges, since  the partition contains one copy  of   the empty set.  Genus  two is indicated by schematically showing two handles. Right: stretching out labelled  edges into corresponding horizontal intervals to produce a  morphism in   $\tfs$.}
\label{fig1_8}
\end{center}
\end{figure}

Vice versa, suppose given $(\ell,g)$ as above, a locally cyclic partition $\lambda\in D^m_{n,cyc}(\ell)$ of $\N^m_n$  into  $\ell+1$ subsets, possibly with some subsets empty, with a cyclic order on each subset. To such data we can assign a connected thin flat surface $S(\lambda,g)$ of  genus $g$ with the horizontal boundary these $n+m$ intervals, $\ell+1$ boundary components, and horizontal  intervals placed according to  the cyclic  order for the subset along each component.  

For another  example, for $n=4$, $m=3$, the partition $\{(1',4,2',2),(3',1,3),()\}$, which includes one copy  of the  empty set, with cyclic orders as indicated and genus $g=2$  
the resulting tf-surface is  shown 
 in Figure~\ref{fig1_8} right. 

\vspace{0.1in}

This bijection between connected  morphisms from $n$ to  $m$ and elements of the set $D^m_{n,cyc}\times\Z_+$ leads  to  a  classification of morphisms in  $\tfs$.
An  arbitrary morphism $S\in \Hom_{\tfs}(n,m)$ is the  union  of the viewable subcobordism of  $S$ and  the floating subcobordism.  The latter are classified by elements of  $\Hom_{\tfs}(0,0)$ and admit  a very explicit  description, via pairs $(\ell,g)$ of the number of circles minus one and the genus  of  each  connected  component. The viewable subcobordism $S'$ of $S$ determines a  partition of $\N^m_n$ by with the set of horizontal  intervals for each component of $S'$ being a part of that  partition. Each part of  this partition is non-empty. 

Next, for each part of the  partition, remove the connected components of $S'$ for all other  parts, downsizing to just  one  component  $S''$. Relabel the horizontal intervals for $S''$ into $1,2,\dots, n^{\prime \prime}$ and  $1,2,\dots, m^{\prime\prime}$. Then such components  $S''$ are classified by data as above: a locally cyclic partition of $\N^{m''}_{n''}$ (possibly with empty subsets included) and a choice of  genus  $g\ge 0$.  

Putting the steps of this algorithm together gives a classification of  morphisms from $n$ to $m$ in $\tfs$. 

\vspace{0.1in}


\subsection{Endomorphisms  of 
\texorpdfstring{$1$}{1} and  homs  between \texorpdfstring{$0$}{0} 
and \texorpdfstring{$1$}{1} 
in \texorpdfstring{$\tfs$}{TFS}}

$\quad$
\vspace{0.1in} 

The category $\tfs$ is rigid symmetric monoidal, with the unit object $0$ and  the generating self-dual object $1$, with all objects being tensor powers  of the generating object, $n=1^{\otimes n}$. 

In the rest of this section, since we only consider the category $\tfs$, we may write  $\Hom(n,m)$ instead of $\Hom_{\tfs}(n,m)$, $\End(n)$ instead of $\End_{\tfs}(n)$, etc. 

\vspace{0.1in} 

{\it Connected endomorphisms of $1$:}
Endomorphisms $\End(1)=\End_{\tfs}(1)$  of the object $1$ in the category $\tfs$ constitute  a monoid. Consider the submonoid $\End^c(1)$ of $\End(1)$ that  consists of  connected endomorphisms of $1$. 
Define endomorphisms $b_1,b_2,b_3\in  \End^c(1)$ via diagrams in  Figure~\ref{fig2_1}. 

\begin{figure}[h]
\begin{center}
\psfrag{b1}{$b_1$}
\psfrag{b2}{$b_2$}
\psfrag{b3}{$b_3$}
\includegraphics[scale=0.75]{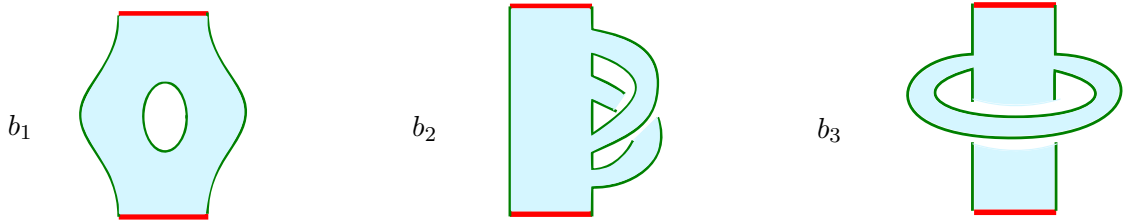}
\caption{Endomorphisms $b_1,b_2,b_3$.}
\label{fig2_1}
\end{center}
\end{figure}

Note that $b_2$ has equivalent presentations, as shown in Figure~\ref{fig2_2}. 
The last diagram is not a  tf-surface, but describes a diffeomorphism   class of  one (rel boundary). The tf-cobordism  $b_2$ has genus one and one  boundary component, with two horizontal segments labelled $1$ and  $1'$ on it, which uniquely determines  it as  an  element of  $\End(1)$.  

\begin{figure}[h]
\begin{center}
\includegraphics[scale=0.7]{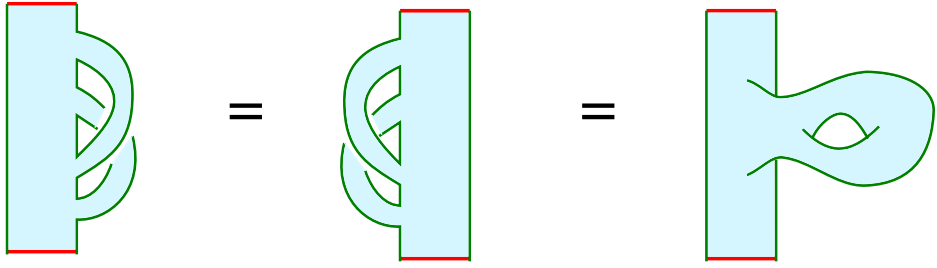}
\caption{Presentations of $b_2$. The diagram  on the right is  not a thin flat presentation but shows a cobordism that can be deformed to a diagram in $\tfs.$ }
\label{fig2_2}
\end{center}
\end{figure}

We refer to $b_1$ as the ``hole'' cobordism, $b_2$ as the ``handle'' cobordism, $b_3$ as the ``cross'' cobordism. 

\begin{prop} The endomorphisms $b_1,b_2,b_3\in\End_{\tfs}(1)$ pairwise commute: 
 \begin{equation*}
        b_1 b_2 = b_2 b_1,\  b_1 b_3 = b_3 b_1, \ b_2  b_3  = b_3 b_2.
    \end{equation*}
\end{prop} 
\begin{proof} Note that the product with $b_1$ just adds a hole with no horizontal segments on it to a connected cobordism. Product with $b_2$ adds a handle to a connected cobordism. 
\end{proof}

\begin{prop}
  \begin{itemize}
  \item 
  $\End^c(1)$ is  a  commutative monoid generated by commuting elements $b_1,b_2,b_3$ with an additional defining relation
  \begin{equation*}
      b_3^2 = b_1 b_3. 
  \end{equation*}
  \item $\End^c(1)$ consists of the following distinct elements: 
  \begin{equation*}
      b_1^n b_2^m, \  b_1^n b_2^m b_3, \ \  n,m\ge 0.
  \end{equation*}
  \end{itemize} 
\end{prop}

\begin{proof}
A cobordism $S\in \End^c(1)$ is a  connected surface with $\ell+1$ boundary circles, genus $g$, and two horizontal intervals on it. If the  intervals are  on the  same connected  component of the boundary, $S=b_1^{\ell}b_2^g$. 
If the intervals lie  on  distinct  boundary components then  $\ell\ge 1$ and $S=b_1^{\ell-1}b_2^gb_3.$
\end{proof}

\vspace{0.1in}

{\it Spaces $\Hom(0,1)$ and  $\Hom(1,0)$:}
An element $y \in  \Hom(0,1)$ is a tf-cobordism with one horizontal interval, at the top. It is a    product $y_1 y_0$ of  one viewable component  $y_1\in \Hom(0,1)$  and a  closed cobordism $y_0\in \Hom(0,0)$.
Assume  that  $y$ is viewable, thus connected, since it has a unique horizontal  segment. Then $y$  is determined  by  the number $\ell+1$ of its boundary components and the genus  $g$ and can be written as
\begin{equation*}
    y = b_1^{\ell}b_2^g \iota,  
\end{equation*}
where  $\iota$ is the morphism $0\lra 1$ shown in  Figure~\ref{fig1_3} on far  left. 
Note that $b_3 \iota = b_1 \iota$. 
\begin{prop}\label{prop_y} A morphism $y\in \Hom_{\tfs}(0,1)$ has a  unique presentation  $y= b_1^{\ell} b_2^g \iota \cdot y_0$, where  $y_0\in \End(0)$ is a floating cobordism. 
\end{prop} 

Reflecting cobordisms about  the  horizontal  line, we obtain a  classification of elements  in $\Hom_{\tfs}(1,0)$.

\begin{prop} A morphism $y\in \Hom_{TFS}(1,0)$ has a  unique presentation  $y= y_0 \cdot \epsilon  b_1^{\ell} b_2^g$, where  $y_0\in \End(0)$ is a floating cobordism. 
\end{prop} 

\vspace{0.1in}

{\it Endomorphism  monoid $\End(1)$.}
Recall that we continue with a minor abuse of  notation, where  we denote by $1$ the generating object of $\tfs$, also use it as the label for the bottom left horizontal interval of a cobordism in $\Hom(n,m)$, and use it convenionally as the label for the first natural number. 

An element $y$ of $\End_{\tfs}(1)$ may be one of the two types: 
\begin{enumerate}
    \item[1).] Horizontal intervals $1$ and $1'$ belong to  the same  connected component of $y$.
    \item[2).] Intervals $1$ and $1'$ belong to  different connected components of $y$. 
\end{enumerate}
Denote by $U_i$ the  set of elements of type $i\in \{1,2\}$, so  that
\begin{equation}\label{eq_disj_union}
  \End(1) = U_1 \sqcup  U_2 \, . 
\end{equation} 
The set $U_2$ is closed  under left and  right multiplication by elements of 
$\End(1)$, thus  constitutes a 2-sided ideal in this  monoid. The set $U_1$ is a unital submonoid in  $\End(1)$. These maps
\begin{equation*}
  U_1 \lra \End(1) \longleftarrow U_2 
\end{equation*}
upgrade decomposition (\ref{eq_disj_union}). 
The monoid $U_1$ is commutative and naturally decomposes 
\begin{equation*} 
U_1 \cong  \End^c(1) \times \End(0)
\end{equation*} 
into the  direct product, both  terms of  which we  have already described.  The direct product corresponds to splitting an element of  $U_1$ into the viewable  connected component and  a floating cobordism. 

Likewise, an element $y$ of $U_2$ splits into a floating cobordism $y_0$ and a viewable one $y_1$. 
A viewable element $y_1$ of $U_2$ consists of two connected components, one bounding horizontal inteval $1$, the other  bounding $1'$. Such an  element  can be written  as 
\begin{equation*}
    y_1 = b_1^{\ell_1}b_2^{g_1}\iota \, \cdot \, \epsilon b_1^{\ell_2}b_2^{g_2} ,
\end{equation*}
with a general $y\in U_2$ given by 
\begin{equation*}
    y = b_1^{\ell_1}b_2^{g_1}\iota \, \cdot y_0 \cdot  \, \epsilon b_1^{\ell_2}b_2^{g_2} .
\end{equation*}
Multiplication of two viewable elements as above produces an additional connected component, see Figure~\ref{fig2_3}, where by the $(\ell,g)$ coupon  we denote the endomorphism $b_1^{\ell}b_2^g$  of $1$.  
\begin{figure}[h]
\psfrag{l1g1}{$l_1,g_1$}
\psfrag{l2g2}{$l_2,g_2$}
\psfrag{l3g3}{$l_3,g_3$}
\psfrag{l4g4}{$l_4,g_4$}
\psfrag{l5g5}{$l_2+l_3,g_2+g_3$}
\begin{center}
\includegraphics[scale=0.75]{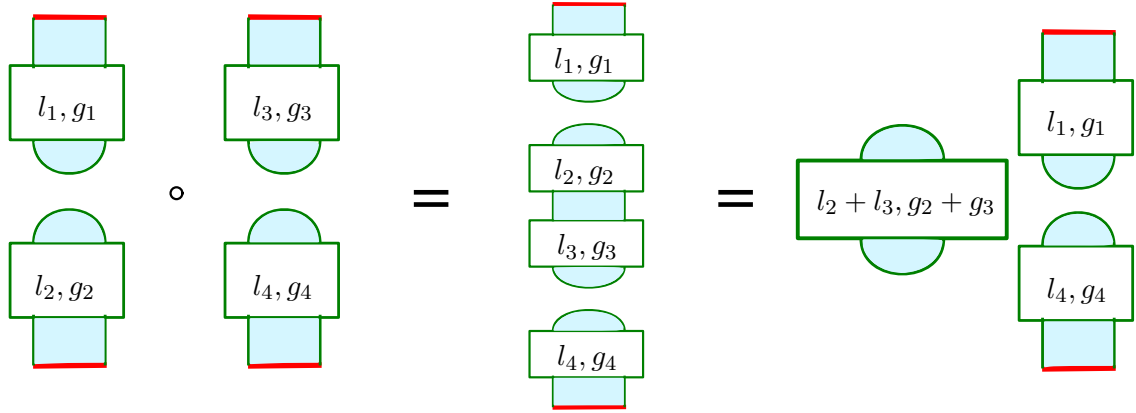}
\caption{Product of two viewable elements  of $U_2$ produces a floating component $\epsilon\, b_1^{\ell_2+\ell_3}b_2^{g_2+g_3}\iota=S_{\ell_2+\ell_3+1,g_2+g_3}$, in addition to the components bounding intervals $1$ and  $1'$.}
\label{fig2_3}
\end{center}
\end{figure}

\vspace{0.1in} 

\emph{Remark:} 
Unlike the monoids $\End(0)$, $\End^c(1)$, and  their  direct product $U_1$, monoid $\End(1)$ and its subsemigroup $U_2$ are not commutative. 

\vspace{0.1in}

%
%

\section{Linearizations of the category \texorpdfstring{$\tfs$}{TFS}}\label{sec_linear}

In this section we work over a  field $\kk$, but  the construction  and  some results may be generalized to an arbitrary commutative ring $R$ (or a commutative ring with  additional conditions, such as being noetherian). A  definitive  starting reference for recognizable series with coefficients in commutative rings is Hazewinkel~\cite{Ha}.

\vspace{0.1in} 


\subsection{Categories \texorpdfstring{$\ktfs$}{kTFS} and \texorpdfstring{$\vtfs_{\alpha}$}{VTFS(a)} for  recognizable \texorpdfstring{$\alpha$}{a}}

$\quad$

\vspace{0.1in} 

{\it Category $\ktfs$.}
Starting with $\tfs$ we can pass  to its preadditive closure $\ktfs$. Objects of $\ktfs$ are the same as those  of $\tfs,$ that  is, non-negative integers $n\in \Z_+$. A morphism in $\ktfs$ from $n$ to $m$ is a finite $\kk$-linear combination of morphisms from $n$ to $m$ in $\tfs$. In particular, $\Hom_{\ktfs}(n,m)$ is a $\kk$-vector space with  a basis $\Hom_{\tfs}(n,m)$. Composition of morphisms is  defined in the obvious   way. 

Category $\ktfs$  is a $\kk$-linear preadditive category. It is also a   rigid symmetric monoidal category. 

\vspace{0.1in}

{\it Power series $\alpha$.}
The ring $\Hom_{\ktfs}(0,0)$ of endomorphisms of the unit object $0$ of $\ktfs$ is naturally isomorphic to the monoid algebra of $\Hom_{\tfs}(0,0).$ The  latter is a free commutative monoid  on generators $S_{\ell+1,g}$, over all $\ell,g\in\Z_+$, so that  
\begin{equation*}
    \Hom_{\ktfs}(0,0) \cong \kk[S_{\ell+1,g}]_{\ell,g\in \Z_+}
\end{equation*}
is the polynomial algebra on countably many generators, parametrized by pairs $(\ell,g)$ of non-negative integers. Homomorphisms of $\kk$-algebras
\begin{equation*}
    \Hom_{\ktfs}(0,0)  \lra \kk
\end{equation*}
are in a bijection with doubly-infinite sequences
\begin{equation*}
    \alpha = (\alpha_{\ell,g})_{\ell,g\in\Z_+}, \ \ \alpha_{\ell,g}\in \kk. 
\end{equation*}
The bijection associates to a sequence $\alpha$ the homomorphism, also denoted  $\alpha$, 
\begin{equation*}
    \Hom_{\tfs}(0,0)\cong \kk[S_{\ell+1,g}]_{\ell,g\in \Z_+}\stackrel{\alpha}{\lra} \kk, \ \  \alpha(S_{\ell+1,g})=\alpha_{\ell,g}.
\end{equation*}
Sequences $\alpha$ are also in a bijection with \emph{multiplicative} $\kk$-valued evaluations of  floating cobordisms in $\tfs$. These evaluations  are maps from the  set of floating cobordisms (endomorphisms of object $0$) in $\tfs$ to  $\kk$ that take disjoint  union of cobordisms  to the product  of evaluations, 
\begin{equation*}
    \alpha(S\sqcup S') = \alpha(S) \cdot \alpha(S').
\end{equation*}

Thus, $\alpha$ is a map of sets 
\begin{equation*}
    \alpha: \Z_+ \times \Z_+ \lra \kk 
\end{equation*}
that we can think  of  a $\Z_+\times\Z_+$-matrix with coefficients in $\kk$

\begin{equation*}
     \alpha=
  \left( {\begin{array}{ccccc}
   \alpha_{0,0} & \alpha_{0,1} & \alpha_{0,2} & \alpha_{0,3} & \dots \\
   \alpha_{1,0} & \alpha_{1,1} & \alpha_{1,2} & \alpha_{1,3} & \dots \\
   \alpha_{2,0} & \alpha_{2,1} & \alpha_{2,2} & \alpha_{2,3} & \dots \\
   \alpha_{3,0} & \alpha_{3,1} & \alpha_{3,2} & \alpha_{3,3} & \dots \\
   \vdots & \vdots & \vdots & \vdots & \ddots \\
  \end{array} } \right)
\end{equation*}
We encode $\alpha$ into  power series in two variables $T_1,T_2$: 
\begin{equation}\label{eq_Z_rat_1}
    Z_{\alpha}(T_1,T_2) = \sum_{k,g\ge 0}\alpha_{k,g} T_1^k T_2^g, \ \ 
     \alpha=(\alpha_{k,g})_{k,g\in \Z_+}, \ \ \alpha_{k,g} \in \kk . 
\end{equation} 
A doubly-infinite sequence $\alpha$ can also be thought of as a linear functional on the space of polynomials in two  variables: 
\begin{equation*}
    \alpha\in \kk[T_1,T_2]^{\ast} := \Hom_{\kk}(\kk[T_1,T_2],\kk).
\end{equation*} 
We assume that $\alpha$ is not identically zero (the theory is trivial otherwise). Then $\ker(\alpha)\subset \kk[T_1,T_2]$ is a codimension one subspace. 

\vspace{0.1in} 

{\it Category $\vtfs_{\alpha}$.}  Given $\alpha$, we can form the quotient $\vtfs_{\alpha}$ of  category $\ktfs$ by adding the relation that  a floating surface  $S_{\ell+1,g}$ of genus $g$ with $\ell+1$ boundary components evaluates to  $\alpha_{\ell,g}\in \kk$. Objects  of  $\vtfs_{\alpha}$ are still non-negative  integers $n$. Morphisms from  $n$ to $m$ are finite $\kk$-linear combinations of \emph{viewable} cobordisms from $n$ to $m$. 
Composition of cobordisms from $n$ to $m$ and from $m$  to  $k$ is  a cobordism from  $n$  to $k$  which  may have  floating  components.
These components are removed simultaneously with multiplying the viewable cobordism that  remains by the product of $\alpha_{\ell,g}$'s, for every component $S_{\ell+1,g}$. 

The space of  homs from $n$  to  $m$  in this category  has a  basis  of  viewable  cobordisms  from $n$ to  $m$. Letter V in the notation $\vtfs_{\alpha}$ stands for \emph{viewable}. 

\vspace{0.1in} 

{\it Recognizable series.}
Borrowing terminology from control theory~\cite{F,FM}, we say that a linear functional or series $\alpha$ is 
\emph{recognizable} if $\ker(\alpha)$ contains an ideal $I\in \kk[T_1,T_2]$ of finite codimension. 

\begin{prop} \label{prop_recog_alpha}
$\alpha$ is recognizable iff  the power series $Z_{\alpha}$ has the form
\begin{equation}\label{eq_recog}
    Z_{\alpha}(T_1,T_2) = \frac{P(T_1,T_2)}{Q_1(T_1)Q_2(T_2)},
\end{equation}
where $Q_1(T_1),Q_2(T_2)$ are one-variable polynomials and $P(T_1,T_2)$ is a two-variable polynomial, all with coefficients in the field $\kk$.  
\end{prop} 

We assume that $Q_1(0)\not=0,$ $ Q_2(0)\not=0$, otherwise at least one of these  polynomials is not  coprime with $P(T_1,T_2)$ and  either $T_1$ or  $T_2$ cancels out from the numerator and denominator. With the  denominator not  zero at  $T_1,T_2=0$ the  power series expansion makes sense. 

\begin{proof}
 See~\cite{F} for a proof. This result is also mentioned in~\cite[Remark 2]{FM}. To prove it, assume that $\alpha$ is recognizable.  We start with the case when $\kk$ is algebraically closed. A finite codimension ideal $I\subset \kk[T_1,T_2]$ necessarily contains a sum,  
 \begin{equation}\label{eq_sum_1}
     I_1 \otimes \kk[T_2]+ \kk[T_1]\otimes I_2\subset I  \subset \kk[T_1,T_2]
 \end{equation}
 for some finite codimension ideals $I_1\subset \kk[T_1]$ and  $I_2\subset  \kk[T_2]$. To see this, note  that the finite affine  scheme $\mathrm{Spec}(\kk[T_1,T_2]/I)$ is supported over finitely many points of the affine plane $\mathbb{A}^2$. Projecting these points onto the coordinate lines and counting  them with multiplicities produces two one-variable polynomials $U_1(T),U_2(T)$ such that $I$ contains the ideal $(U_1(T_1))+(U_2(T_2))$ of $\kk[T_1,T_2]$. We can now take principal ideals  $I_i=(U_i(T_i))$, $i=1,2$ to get the inclusion on the LHS of  (\ref{eq_sum_1}). This also gives a quotient map 
 \begin{equation*}
     \kk[T_1]/(U_1(T_1)) \otimes  \kk[T_2]/(U_2(T_2)) \lra \kk[T_1,T_2]/I 
 \end{equation*}
 lifting to the identity map on $\kk[T_1,T_2]$. Existence of such finite codimension ideals $I_1,I_2$ over an arbitrary  field  $\kk$ follows as well. 
 
 Hence, recognizable  series  $\alpha$  has  the  property that  $\alpha(U_1(T_1)T_1^k T_2^m)=0$ for any $k,m\ge 0$. We can assume that $U_1(T)$ is a polynomial of some  degree $r$ with the  lowest degree term $u_sT^s$ for $s\le  r$ and write 
 \begin{equation*}
     U_1(T) = u_r T^r + u_{r-1}T^{r-1}+\dots + u_{s+1} T^{s+1}+u_sT^s, \ 0\le s \le r,  \ u_r,u_s\not= 0, \ u_j \in \kk.
 \end{equation*}
 Then, for any $k,m\ge 0$ 
 \begin{equation} \label{eq_ua_1}
 u_r\alpha_{r+k,m}+ u_{r-1}\alpha_{r-1+k,m}+\dots + u_{s+1}\alpha_{s+k+1,m}+ u_s\alpha_{s+k,m} = 0.
 \end{equation} 
 We obtain  a similar relation on the  coefficients with $U_2$ and  $T_2$ in place of $U_1$ and $T_1$  and varying  the second  index. Let us write  
 \begin{equation*}
     U_2(T) = v_{r'} T^{r'} + v_{r'-1}T^{r'-1}+\dots + v_{s'+1} T^{s'+1}+v_{s'}T^{s'}, \ 0\le s' \le r',  \ v_{r'},v_{s'}\not= 0, \ v_j \in \kk.
 \end{equation*}
 Then, for any $k,m\ge 0$ 
 \begin{equation} \label{eq_ua_2}
 v_{r'}\alpha_{r'+k,m}+ v_{r'-1}\alpha_{r'-1+k,m}+\dots + v_{s'+1}\alpha_{s'+k+1,m}+ v_s\alpha_{s'+k,m} = 0.
 \end{equation} 
 
 Consequently,  $\alpha$ is eventually recurrent in both $T_1$ and $T_2$ directions and  its values are  determined by 
 $\alpha_{i,j}$ with $0\le i <r, 0\le j< r'$. 
 
 Consider polynomials 
 \begin{eqnarray*}
     \widehat{Q}_1(T) &  = &  T^r U_1(T^{-1}) = u_s T^r + u_{s+1}T^{r-1}+u_{s+2}T^{r-2}+\dots + u_r T^{r-s}, \\
     \widehat{Q}_2(T) &  = &  T^{r'} U_2(T^{-1}) = v_{s'}T^{r'} + v_{s'+1}T^{r'-1}+v_{s'+2}T^{r'-2}+\dots + v_{r'} T^{r'-s'}.
 \end{eqnarray*}
 Form the product  
 \begin{equation*}
   \widehat{P}(T_1,T_2):=Z_\alpha(T_1,T_2)\widehat{Q}_1(T_1)\widehat{Q}_2(T_2) = \sum_{i,j\ge 0} w_{i,j} T_1^i T_2^j
 \end{equation*}
and  examine coefficients of its power series expansion. Formulas  (\ref{eq_ua_1}), (\ref{eq_ua_2}) show that  $w_{i,j}=0$ if $i\ge r$ of $j\ge r'$. Therefore, $\widehat{P}(T_1,T_2)$ is a polynomial with $T_1,T_2$ degrees bounded by $r-1$, $r'-1$, respectively. We can  then form the quotient 
\begin{equation*} 
  \frac{\widehat{P}(T_1,T_2)}{\widehat{Q}_1(T_1)\widehat{Q}_2(T_2)}
  \end{equation*}
The numerator and denominator may share common factors, including $T_1^{r-s}T_2^{r'-s'}$. After canceling those out, we arrive at the presentation (\ref{eq_recog}) for $Z_{\alpha}(T_1,T_2)$. 
  
We leave the proof of the opposite  implication of the proposition to the reader or refer to~\cite{F}. 
  
Note that the proof works for any finite  number of  variables $T_1,\dots,  T_c$, not only for two. 
\end{proof}

The condition that $\alpha$ is  \emph{recognizable} can also  be expressed via its Hankel matrix $H_{\alpha}$.  The latter matrix has rows and  columns enumerated by pairs $(m,k)\in  \Z_+\times\Z_+$,  equivalently by  the monomial basis elements $T_1^mT_2^k$.  The $((m_1,k_1),(m_2,k_2))$-entry  of $H_{\alpha}$ is $\alpha_{m_1+m_2,k_1+k_2}$.  The following result is proved in~\cite{F}. 

\begin{prop} The series  $\alpha$ is recognizable iff the Hankel matrix $H_{\alpha}$ has finite rank. 
\end{prop} 
Note that $H_{\alpha}$ has finite rank iff there exists $M$ such that any $M\times M$ minor of $H_{\alpha}$ has determinant zero. The rank is  $M-1$  if  in  addition there is an $(M-1)\times (M-1)$ minor with a non-zero determinant.   

\vspace{0.1in}


\subsection{Skein category \texorpdfstring{$\stfs_{\alpha}$}{STFS(a)}.} 
\label{subset_skein}   
$\quad$

\vspace{0.1in} 

{\it Recognizable series and commutative Frobenius algebras.} 
Assume  that $\alpha$ is recognizable. Among all finite-codimension ideals $I\subset\ker(\alpha)$ there is a unique largest ideal $I_{\alpha}$, given by the sum over  all such $I$. Equivalently, it can  be described as follows. There  is a homomorphism of $\kk[T_1,T_2]$-modules  
\begin{equation}\label{eq_h_hom}
    h :  \kk[T_1,T_2] \lra \kk[T_1,T_2]^{\ast}
\end{equation}
given by sending  $1$ to $\alpha$  and $z\in\kk[T_1,T_2]$ to $z\alpha\in \kk[T_1,T_2]^{\ast}$ with $(z\alpha)(f)=\alpha(zf)$. The ideal $I_{\alpha}$ is the kernel  of $h$.  

Notice that $\alpha$ descends to  a  nondegenerate bilinear  form on the   quotient algebra 
\begin{equation}\label{eq_A_alpha}
    A_{\alpha} \ := \  \kk[T_1,T_2]/I_{\alpha}.
\end{equation}
In particular, $A_{\alpha}$ is a commutative Frobenius algebra on two  generators $T_1,T_2$  with a nondegenerate trace  form $\alpha$. 

Vice versa, assume  given a commutative Frobenius $\kk$-algebra $B$ with the nondegenerate trace form $\beta:B\lra \kk$ and a pair  of generators $g_1,g_2$. To such data we can associate  a surjective homomorphism 
\begin{equation*}
   \psi: \kk[T_1,T_2] \lra B,\  \psi(T_i)=g_i, \ i = 1,2, 
\end{equation*}
the trace map $\alpha=\beta\circ   \psi$ on $\kk[T_1,T_2]$ given by composing  $\psi$  with $\beta$, 
and recognizable series 
\begin{equation*}
    \alpha_{\beta}= \sum_{\ell,g\ge 0} \beta(g_1^{\ell}g_2^g) T_1^{\ell}T_2^g . 
\end{equation*}
Thus, recognizable power series on  $\kk[T_1,T_2]$ are classified by isomorphism classes  of data $(B,g_1,g_2,\beta)$: a commutative Frobenius algebra $B$  generated by  $g_1,g_2\in B$ and a non-degenerate trace $\beta$. 

\vspace{0.1in} 

{\it Category $\stfs_{\alpha}$.}
We can now define the category $\stfs_{\alpha}$ (where first S stands for ``skein'') to be a quotient of  $\vtfs_{\alpha}$ by the skein relations  in the ideal $I_{\alpha}$. The category $\stfs_{\alpha}$ has the same objects as all the other cobordism categories  we've  considered so far,  that is, nonnegative  integers $n$. Morphisms from $n$  to $m$ are $\kk$-linear combinations of  viewable cobordisms modulo the relations in $I_{\alpha}$. Precisely, let 
\begin{equation}\label{eq_poly_p}
    p(T_1,T_2)=\sum_{i,j}  p_{i,j}T_1^i T_2^j \in I_{\alpha}
\end{equation}
be a polynomial in the ideal $I_{\alpha}$. Given a viewable cobordism $x$ choose a component $c$ of $x$ and denote by  $x_c(i,j)$ the cobordism  given by  inserting  $i$ holes and adding $j$  handles to $x$ at  the component $c$. We  now mod out the  hom  space  $\Hom_{\vtfs_{\alpha}}(n,m)$, which is a  $\kk$-vector space with a basis of all viewable cobordisms from $n$ to  $m$,  by  the relations 
\begin{equation*}
\sum_{i,j}  p_{i,j}\, x_c(i,j)=0, 
\end{equation*}
one  for each component $c$ of  $x$, over all viewable cobordisms $x$.

It is  easy to see that these ``skein''  relations are compatible with $\alpha$-evaluation of floating cobordisms. Namely, if instead of a viewable cobordism  $x$ we consider  a floating cobordism  $y$ and  choose a  component $c$ of  $y$  to add holes and handles,  resulting in cobordisms $y_c(i,j)$, then 
\begin{equation*}
\sum_{i,j}  p_{i,j}\alpha(y_c(i,j))=0. 
\end{equation*}
This compatibility condition, immediate from  our definition of  $I_{\alpha}$ as the kernel of the module map (\ref{eq_h_hom}), ensures non-triviality of this quotient  and its  compatibility with the composition of morphisms.

\vspace{0.1in} 

Viewing  $\vtfs_{\alpha}$ as a tensor category, it is enough to write down corresponding relations on homs from $0$ to $1$ and then mod  out by them in the tensor category (by gluing  each term in the  resulting linear  combination of products  of  holes  and  handles on a disk to any component along a segment on its side  boundary).  Choose a generating set $v_1,\dots,v_r$  of  $I_{\alpha}$ viewed as $\kk[T_1,T_2]$-module. Specializing to  a single   basis element $v_j$, assume that it is given by the polynomial $p$  on the right hand side  of (\ref{eq_poly_p}). Form the element 
\begin{equation*}
    b(v_j) := \sum_{i,j} p_{i,j} b_1^i b_2^j \, \iota \in \Hom(0,1). 
\end{equation*}

The skein  category $\stfs_{\alpha}$ can  be defined as the quotient of  $\vtfs_{\alpha}$ by the tensor ideal generated by elements $b(v_1),\dots, b(v_r)$. Figure~\ref{fig3_1}  shows an example of an element  $b(v)$.

\begin{figure}[h]
\begin{center}
\includegraphics[scale=0.6]{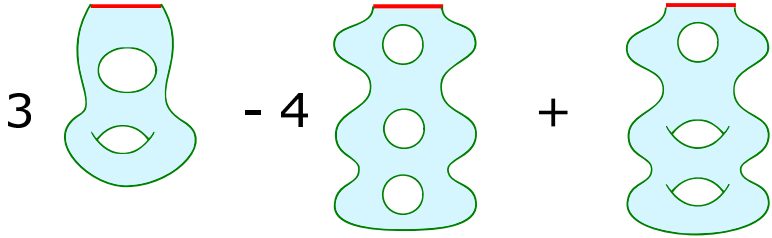}
\caption{$b(v)$, for $v=3 T_1 T_2 - 4 T_1^3 + T_1 T_2^2 $; handles are shown schematically.}
\label{fig3_1}
\end{center}
\end{figure}

\vspace{0.1in} 

{\it Remark:}
For a recognizable series $\alpha$ there are unique minimal degree monic polynomials $q_{\alpha,1},q_{\alpha,2}$, 
\begin{equation*}
    q_{\alpha,1}(x) = x^{t}+a_{t-1}x^{t-1}+\dots + a_0, \ \  q_{\alpha_2}(x) =  
    x^{t'}+a'_{t'-1}x^{t'-1}+\dots + a'_0,
\end{equation*}
such  that 
\begin{equation*} 
q_{\alpha,1}(T_1)\in I_{\alpha}, \ \ 
q_{\alpha,2}(T_2)\in I_{\alpha}.
\end{equation*} 
Among skein relations associated to elements of $I_{\alpha}$ in $\stfs_{\alpha}$ there is a polynomial relation that utilizes only adding holes to a component of the cobordism.  This  relation is  given by the polynomial $q_{\alpha,1}(T_1)$: 
\begin{equation*}
    b_1^{t}+a_{t-1}b_1^{t-1}+\dots + a_0=0,
\end{equation*}
describing an equality in the ring of  endomorphisms of  object $1$ of $\stfs_{\alpha}$, where  $b_1$ is the  \emph{hole} cobordism, see  Figure~\ref{fig2_2}. Equivalently, it can be rewritten as a relation  in $\Hom(0,1)$: 
\begin{equation*}
    (b_1^{t}+a_{t-1}b_1^{t-1}+\dots + a_0)\iota=0,
\end{equation*}

Likewise, 
there  is  a  skein  relation on cobordisms that differ only by genus  of a given component. The  relation  is given by the polynomial $q_{\alpha,2}(T_2)$:
\begin{equation*}
    b_2^{t'}+a'_{t'-1}b_2^{t'-1}+\dots + a'_0 =0, 
\end{equation*}
where  $b_2$ is the \emph{handle} morphism, see Figure~\ref{fig2_2}.

\vspace{0.1in} 

{\it Minimal viewable cobordisms, $B_{\alpha}$-companions, and bases of hom spaces of  $\stfs_{\alpha}.$}

Consider a connected viewable cobordism $x$. 
We say that  $x$ is \emph{minimal} if it has genus zero and no \emph{holes}, that  is, each boundary component of  $x$ contains at least one horizontal segment. Equivalently $x$ is  minimal if it  cannot  be  factored  into $x'b_1 x''$ or $x' b_2 x''$ for some morphisms $x',x''$. Note that if such a factorization exists, then there exists one with $x''$ the identity cobordism and one with $x'$ the identity cobordism. Any viewable connected cobordism $x$ from $n$  to $m$ with $m>0$ can  be written as $(b_1^ib_2^j\otimes \id_{m-1})y$ for some  minimal $y$ and, if $n>0$,  as 
$y(b_1^ib_2^j\otimes \id_{n-1})$  for the  same  $y$, see Figure~\ref{fig3_1_1}. If one of $n$ or $m$ is zero, only  one of these two presentations exist.
\begin{figure}[h]
\begin{center}
\psfrag{x}{$~~x$}
\psfrag{y}{$~~y$}
\psfrag{b1b2}{$b_1^ib_2^j$}
\includegraphics[scale=0.7]{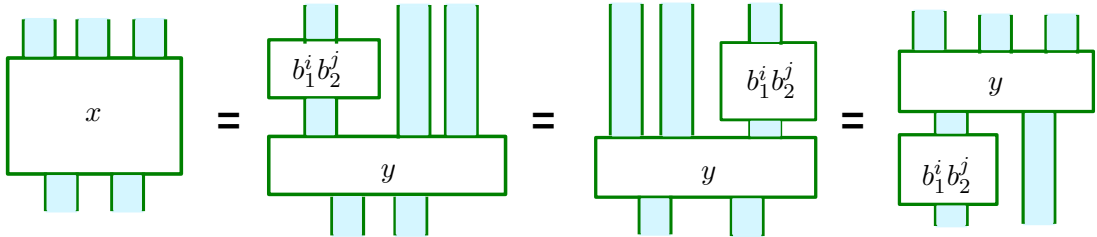}
\caption{Factorization of a connected  cobordism $x$ into a coupon and a minimal cobordism is shown schematically. Since $y$ is  connected, $b_1^i b_2^j$ coupon can be moved to any leg of $y$. }
\label{fig3_1_1}
\end{center}
\end{figure}
Equivalently, a connected viewable cobordism 
$x$ is minimal if it is \emph{handless} and has no holes. 

\vspace{0.1in}

A viewable cobordism $y$ is called \emph{minimal} if each connected component  of $y$ is  minimal. A viewable cobordism $x$ factors into a product of a minimal  cobordism and ``coupons'' carrying powers of $b_1,b_2$, one for each connected component 
of $x$. That is, for each connected component $c$ of $x$ count holes and handles on it and then remove them to get a minimal connected component $c'$. The original component can be recovered by inserting holes and handles back anywhere along $c'$. For instance, they may be inserted at one  of its top or bottom legs by  multiplying $c'$ by the corresponding powers of $b_1$ and  $b_2$ there. 

To any viewable $x$ we can associate its minimal counterpart  $y$  by removing  holes and handles  from each  connected component of $x$. Given $y$, we can recover $x$ by multiplying  by appropriate powers  of $b_1$ and $b_2$ at horizontal intervals for different  components of $y$.  

Denote  by $\mc{M}(n,m)$ the set of minimal viewable cobordisms from $n$ to $m$. 

\begin{prop} $\mc{M}(n,m)$ is a finite set. 
\end{prop}

\begin{proof} From  our  classification of morphisms in  $\tfs$ it is clear that minimal cobordisms from $n$  to $m$ are in a bijection with partitions $\lambda$ of the  set $\N^m_n$ of $n+m$ horizontal intervals, together with a  choice of a partition $\mu_i$ of each part $\lambda_i$ of $\lambda$ and a cyclic order on each  part of $\mu_i$. 
\end{proof} 

Recall finite codimension ideal $I_{\alpha}$ (the \emph{syntactic} ideal) associated with recognizable series $\alpha$. 
Let
\begin{equation*}
    d_{\alpha}=\dim (\kk[T_1,T_2]/I_{\alpha}).
\end{equation*}
 Choose a  set of pairs 
\begin{equation*}
P_{\alpha}= \{(i_t,j_t)\}_{t=1}^{d_{\alpha}}, \ \  i_t,j_t \in \Z_+
\end{equation*} 
such  that monomials $T_1^{i_t}T_2^{j_t}$ constitute a basis of the algebra  $\kk[T_1,T_2]/I_{\alpha}$. Denote this basis by $B_{\alpha}$. It is  well-known~\cite{MiS} that a basis can always be  choosen so that the exponents $(i_t,j_t)$ of  the monomials, when placed into corresponding points of  the square lattice, constitute a partition of $d_{\alpha}$, but we do  not need this result here. 

Choose a  minimal cobordism $y$ and assign an element $v_c\in  B_{\alpha}$ to each connected  component $c$ of $y$. This  assignment  gives rise to a cobordism $x$ obtained from  $y$ by inserting cobordisms  $b(v_c)$ at all  components $c$ of $y$. For $v_c=T_1^iT_2^j$ we add $i$  holes and $j$ handles to the component $c$ or, equivalently,  multiply it at one of its horizontal  boundary  intervals by $b_1^ib_2^j$. 

In this way to $y\in \mc{M}(n,m)$ there are assigned  $d_{\alpha}^r$ cobordisms $x$, where $r $ is the number of components  of $y$. These  $x$ are called $B_{\alpha}$-companions  of $y$. Denote  the set of  such $x$ by $B_{\alpha}(y). $

\begin{prop} Elements of sets $B_{\alpha}(y)$, over all $y\in \mc{M}(n,m)$,  constitute a basis of $\Hom_{\stfs_{\alpha}}(n,m).$
\end{prop}

In other words, to get a basis of homs from  $n$ to  $m$ in the skein category $\stfs_{\alpha}$ we take all minimal cobordisms $y$ from $n$ to $m$ and  insert a basis element from  $B_{\alpha}$ into each component of $y$. 

\begin{proof} 
The proposition follows immediately  from our construction of $\stfs_{\alpha}$. One needs to check consistency, that our rules do not force additional relations when composing cobordisms. This is  straightforward. 
\end{proof} 

\begin{cor}
 Hom spaces  in the category $\stfs_{\alpha}$ are finite dimensional. 
\end{cor}

\vspace{0.1in}

{\it Remark:} In a seeming discrepancy, object $1$ of the category  $\tfs$ is a symmetric Frobenius object but not a commutative Frobenius object, see Figure~\ref{fig1_5_1} left, since the multiplication map  $1\otimes  1\lra  1$ does not commute with  the  permutation endomorphism of  $1\otimes 1$. 
Yet, in  the category $\stfs_{\alpha}$ the state space $\Hom(0,1)$ of  the interval is a commutative Frobenius algebra $A_{\alpha}$, 
defined in (\ref{eq_A_alpha}), with the multiplication on $\Hom(0,1)$ given by the  \emph{thin flat pants} cobordism in Figure~\ref{fig3_2} left.  This  is explained by the  observation that  the thin  flat pants multiplication is  commutative  in the  categories we consider, including $\tfs$  and  $\vtfs_{\alpha}$  and $\stfs_{\alpha}$. Indeed, viewable morphisms from  $0$ to $1$ in $\tfs$ have the form $b_1^n b_2^m \iota$, and the product of two such morphisms does not depend on  their order, see  Figure~\ref{fig3_2} right. Adding floating  components (or passing to  linear combinations, or taking quotients) does not break commutativity.

Later, in Section~\ref{subsec_adding}, in a similar situation we  also denote $b_1^n b_2^m \iota$ by  $\undb_1^n\undb_2^m$.

\begin{figure}[h]
\psfrag{c1}{$b_1^{n_1}b_2^{n_2}$}
\psfrag{c2}{$b_1^{m_1}b_2^{m_2}$}
\psfrag{c3}{$b_1^{n_1+m_1}b_2^{n_2+m_2}$}
\begin{center}
\includegraphics[scale=0.8]{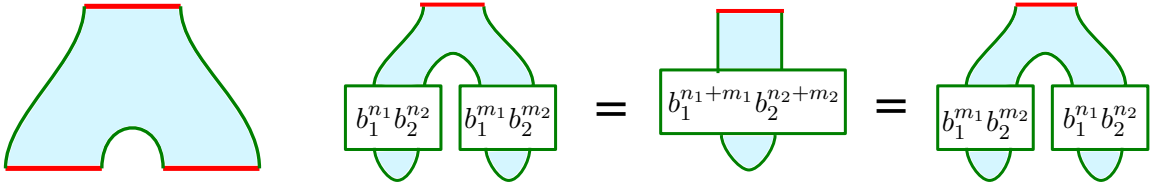}
\caption{Left: thin flat pants  cobordism  from $1\otimes 1$ to  $1$. Right: commutativity of multiplication in  $\Hom(0,1)$. }
\label{fig3_2}
\end{center}
\end{figure}

\vspace{0.1in} 


\subsection{Quotient by negligible  morphisms and  Karoubi envelopes.}

$\quad$

\vspace{0.1in}

{\it Category  $\tfs_{\alpha}$.} Consider the ideal $J_{\alpha}\subset \stfs_{\alpha}$ of negligible morphisms, relative to the trace form $\tral$ associated with $\alpha$, and form the quotient category  
\begin{equation*}
    \tfs_{\alpha} := \stfs_{\alpha}/J_{\alpha}.
\end{equation*}

The trace form is given on a cobordism $x$ from $n$ to $n$ by closing it via $n$ annuli connecting $n$ top with $n$ bottom circles of the horizontal boundary  of $x$ into a floating cobordism $\widehat{x}$ and  applying $\alpha$, 
\begin{equation*}
  \tral(x) := \alpha(\widehat{x}). 
\end{equation*} 
This operation is depicted in  Figure~\ref{fig3_3}. 

\begin{figure}[ht]
\psfrag{x}{$~~x$}
\psfrag{xh}{$\widehat{x}~=$}
\psfrag{al}{$\alpha$}
\psfrag{axh}{$\alpha(\widehat{x})$}
\begin{center}
\includegraphics[scale=0.65]{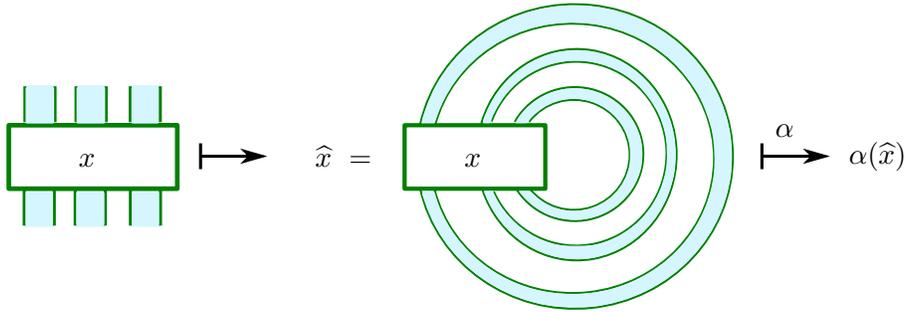}
\caption{The trace map: closing endomorphism $x$ of $n$ into  $\widehat{x}$ and applying $\alpha$.  }
\label{fig3_3}
\end{center}
\end{figure}
 
 A morphism $y\in\Hom(n,m)$ is called \emph{negligible} if $\tral(zy)=0$ for any morphism $z\in\Hom(m,n)$. Negligible morphisms constitute a two-sided ideal in the pre-additive category $\stfs_{\alpha}$. 
 
 The quotient category $\tfs_{\alpha}$ has finite-dimensional hom spaces, as does  $\stfs_{\alpha}$ (recall that $\alpha$ is recognizable). The trace is nondegenerate on $\tfs_{\alpha}$ and defines perfect bilinear  pairings
 \begin{equation*}
     \Hom(n,m) \otimes \Hom(m,n) \lra \kk
 \end{equation*}
 on its hom spaces. We may call $\tfs_{\alpha}$ the  \emph{gligible quotient} of $\stfs_{\alpha}$, having modded out by  the ideal of negligible morphisms.

 \vspace{0.1in}
 
 Let us go back to the category $\tfs$ and  its linear version  $\kk\tfs$. Fix the number $n$ of intervals  and  consider the vector space $V_n$ with a basis of all viewable tf-surfaces with that boundary, that is  viewable cobordisms in $\tfs$  from $0$  to $n$. Given $\alpha$, define a bilinear form on $V_n$ via its values on pairs of basis vectors: 
 \begin{equation*}
 (x,y) = \alpha(\overline{y}x) \in\kk, 
 \end{equation*}
 where  $\overline{y}$ is given by reflecting  $y$ about a horizontal  line to  get a cobordism from $n$ to $0$, and $\overline{y}x$ is a floating cobordism from  $0$ to $0$ given by composing $\overline{y}$ and $x$. This bilinear form on $V_n$ is symmetric. Define $A_{\alpha}(n)$ as the quotient of $V_n$ by  the  kernel of this  bilinear form. Then there is a  canonical isomorphism
 \begin{equation*}
     A_{\alpha}(n) \cong \Hom_{\tfs_{\alpha}}(0,n)
 \end{equation*}
 as  well as isomorphisms 
 \begin{equation*}
     A_{\alpha}(n+m) \cong \Hom_{\tfs_{\alpha}}(0,n+m)  \cong \Hom_{\tfs_{\alpha}}(m,n)
 \end{equation*}
 given by moving $m$ invervals from bottom to top via the duality morphism. 
 
 \vspace{0.1in} 
 
 The symmetric  group $S_n$ acts by permutation  cobordisms on $A_{\alpha}(n)$. Furthermore, at each circle there is an action of the endomorphism algebra $\End(1)=\End_{\tfs_{\alpha}}(1)$. Consequently, the cross-product algebra $\kk\S_n \ltimes \End(1)^{\otimes k}$ acts on $A_{\alpha}(n)$.
 
 Multiplication  maps 
 \begin{equation*} 
 A_{\alpha}(n) \otimes  A_{\alpha}(m)\lra A_{\alpha}(n+m)
 \end{equation*} 
 turn the direct sum 
 \begin{equation*}
     A_{\alpha} \ := \ \bigoplus_{n\ge 0} A_{\alpha}(n)
 \end{equation*}
 into a unital commutative associative graded algebra, with  $A_{\alpha}(0)\cong \kk$. 
 All  of this  data, including the  power series $\sum_{n\ge 0}\dim A_{\alpha}(n)z^n$ encoding dimensions of  $A_{\alpha}(n)$, are invariants of recognizable series $\alpha$.
 \vspace{0.1in}

In the diagram of five categories and four functors 
\begin{equation*}
    \tfs \lra \kk\tfs\lra \vtfs_{\alpha} \lra \stfs_{\alpha} \lra \tfs_{\alpha}
\end{equation*}
one can get  from $\kk\tfs$ to  $\tfs_{\alpha}$ in one step, bypassing  
$\vtfs_{\alpha}$ and $\stfs_{\alpha}$,  
by taking the ideal of negligible morphisms in $\kk\tfs$ (for essentially the same  trace map, shown in  Figure~\ref{fig3_3}) and modding  out by it. It is convenient to introduce those  intermediate categories, though. For  instance, $\stfs_{\alpha}$ already has finite-dimensional hom spaces  and   allows  to define the analogue of  the Deligne category in our case.

\vspace{0.1in} 

{\it The Deligne category $\dtfs_{\alpha}$ and its gligible  quotient $\udtfs_{\,\alpha}$.}
The skein category  $\stfs_{\alpha}$  is a rigid symmetric monoidal $\kk$-linear category with  objects $n\in \Z_+$ and finite-dimensional hom spaces. 
We form the additive Karoubi  closure 
\begin{equation*}
    \dtfs_{\alpha} \ := \ \Kar(\stfs^{\oplus}_{\alpha})
\end{equation*}
by allowing formal finite direct sums of objects in $\stfs$, extending morphisms correspondingly, and then adding idempotents to get a Karoubi-closed  category. Category $\dtfs_{\alpha}$ plays the role of the Deligne category in  our construction.  

\vspace{0.1in}

In the Deligne category  $\dtfs_{\alpha}$ endomorphisms of an object $(n,e)$, where $e$ is an idempotent  endomorphism of $n$, inherit the trace map $\tr_{\alpha}$ into the ground  field. Consequently, category $\dtfs_{\alpha}$  also has a two-sided ideal  of  negligible morphisms  $J_{D,\alpha}$. Taking the quotient by  this ideal  
\begin{equation*}
    \udtfs_{\,\alpha} \ := \ \dtfs_{\alpha}/J_{D,\alpha}
\end{equation*}
gives us an idempotent-complete category with non-degenerate symmetric bilinear   forms on hom  spaces $\Hom(0,(n,e))$, where $(n,e)$ is an object  as above, and more  generally non-degenerate  bilinear pairings  on hom spaces
\begin{equation*}
    \Hom((n,e),(m,e'))\otimes \Hom((m,e'),(n,e)) \lra  \kk
\end{equation*}
where $e'$ is an  idempotent endomorphism of object $m$. Due to the symmetry between homs given by the contravariant  involution on all categories that we have considered so far (reflection about a horizontal line), the  above  bilinear pairings can be converted into  non-degenerate symmetric bilinear forms  on 
$\Hom((n,e),(m,e'))$ in $\udtfs_{\,\alpha}$. 

Category $\udtfs_{\,\alpha}$ is also equivalent  to the additive  Karoubi closure of  the category  $\tfs_{\alpha}$, see  the  commutative square  in (\ref{eq_seq_cd}).


\subsection{Summary  of the categories  and  functors} \label{subsec_summary} 

Below is a summary for each category  that has been considered. 

\begin{itemize}
    \item  $\tfs$: the category of thin flat  surfaces  (tf-surfaces). Its objects are non-negative integers and morphisms are thin  flat  surfaces. 
    \item $\kk\tfs$: this category has the  same objects as $\tfs$; its  morphisms are formal finite $\kk$-linear combinations of morphisms in $\tfs$. 
    \item  $\vtfs_{\alpha}$: in this  quotient category of $\kk\tfs$ we reduce morphisms to linear combinations  of  viewable cobordisms. Floating  connected components are  removed  by  evaluating them  via  $\alpha$.  
    \item $\stfs_{\alpha}$: to define this category,   specialize to recognizable $\alpha$ and add skein relations by modding out  by elements of the ideal  $I_{\alpha}$ in  $\kk[T_1,T_2]$ along each connected  component of a  surface ($T_1$  is a hole,  $T_2$  a  handle). Hom spaces in this category are finite-dimensional. 
    \item  $\tfs_{\alpha}$: the quotient  of  $\stfs_{\alpha}$ by the  ideal $J_{\alpha}$ of  negligible morphisms.  This category  is also equivalent  (even isomorphic) to the quotients   of $\kk\tfs$ and $\vtfs_{\alpha}$ by the corresponding ideals of negligible morphisms in  them. The trace pairing in $\tfs_{\alpha}$ between $\Hom(n,m)$ and $\Hom(m,n)$ is perfect. 
    \item $\dtfs_{\alpha}$: it is the  analogue  of the Deligne category obtained from $\stfs_{\alpha}$ by allowing finite direct sums of objects and then  adding  idempotents  as objects to get a Karoubi-closed category. 
    \item $\udtfs_{\,\alpha}$: the quotient  of $\dtfs_{\alpha}$ by the two-sided ideal of negligible  morphisms. This category is equivalent  to the additive  Karoubi closure of $\tfs_{\alpha}$ and  sits in  the bottom right corner of the  commutative square below.  
\end{itemize}

We arrange these  categories and   functors, when $\alpha$ is recognizable,
into  the following diagram: 

\begin{equation} \label{eq_seq_cd}
\begin{CD}
\tfs  @>>> \kk\tfs @>>> \vtfs_{\alpha} @>>> \stfs_{\alpha} @>>>  \dtfs_{\alpha} \\
@.   @. @.   @VVV   @VVV  \\    
 @.  @.   @.  \tfs_{\alpha}  @>>>  \udtfs_{\,\alpha}
\end{CD}
\end{equation}
All seven  categories  are rigid symmetric monoidal. All but the leftmost  category $\tfs$ are  $\kk$-linear. Except for the two categories on the far right, the  objects of  each category are non-negative integers. 
The four categories on the right  all  have finite-dimensional hom spaces. The  two categories on the far right are additive and Karoubi-closed. The four categories in the middle of the  diagram  are pre-additive but not additive.  

The arrows show functors between these   categories considered  in the paper. The square is commutative. An analogous diagram of functors and categories can be found in~\cite{KS} for  the  category of oriented 2D cobordisms in  place of $\tfs$.   

It is possible to go directly from $\kk\tfs$ to  $\tfs_{\alpha}$ by modding  out by  the ideal of negligible morphisms  in the former category. We found  it convenient to get to this  quotient  in several steps, introducing  categories  $\vtfs_{\alpha}$ and  $\stfs_{\alpha}$ along the way.

\vspace{0.07in} 
 
\emph{Remark:}  For possible  future use,  it may be  convenient to relabel the categories  above using shorter strings. For instance, writing  $\mc{S}$ (for ``surfaces'') in place of $\tfs$ we can rename the  categories as follows: 
\begin{equation} \label{eq_seq_cd_2}
\begin{CD}
\mc{S}  @>>> \kk\mc{S} @>>> \mc{VS}_{\alpha} @>>> \mc{SS}_{\alpha} @>>>  \mc{DS}_{\alpha} \\
@.   @. @.   @VVV   @VVV  \\    
 @.  @.   @.  \mc{S}_{\alpha}  @>>>  \underline{\mc{DS}}_{\,\alpha}
\end{CD}
\end{equation}

\vspace{0.07in} 

For convenience we wrote below short reminders of what these  categories are: 

\begin{equation}\label{eq_words}
\begin{gathered}
 \xymatrix{
\textrm{cobordisms}\ar[r] & \Bbbk\textrm{-linear} \ar[r] & \textrm{viewable} \ar[r] & \textrm{skein}\ar[d] \ar[r] &
{\textrm{Deligne (Karoubian)}} \ar[d] \\
& & & \textrm{gligible} \ar[r] & 
\textrm{gligible and Karoubian}
}
\end{gathered}
\end{equation}

If $\alpha$  is not recognizable, we can still  define categories $\vtfs_{\alpha}$, $\tfs_{\alpha}$ and $\udtfs_{\,\alpha}$ (in the streamlined notation, categories $\mc{VS}_{\alpha}$, $\mc{S}_{\alpha}$ and $\underline{\mc{DS}}_{\,\alpha}$), but  it is not clear whether  these categories may be  interesting for some such $\alpha$.

\vspace{0.1in} 

%
%

\section{Hilbert scheme and recognizable series}  
\label{sec_sample} 

\vspace{0.1in} 

{\it Recognizable series and points on the dual tautological bundle.}
Recognizable series $\alpha$ gives rise to the ideal $I_{\alpha}$  in $\kk[T_1,T_2]$  of finite codimension $k=d_{\alpha}$  
and the  quotient algebra $A_{\alpha}$  by this ideal, see formula (\ref{eq_A_alpha}) in Section~\ref{subset_skein}. This algebra  is commutative Frobenius and  comes with  two  generators $T_1,T_2$ and a non-degenerate  trace. The ideal $I_{\alpha}$ describes a point in  the Hilbert scheme  of codimension $k$ ideals in $\mathbb{A}^2=\mathrm{Spec}\, \kk[T_1,T_2]$, where 
\begin{equation*}
k=d_{\alpha}=\dim A_{\alpha}.
\end{equation*} 

Let us specialize to the ground field $\kk=\C$.  
 Denote by $\Rec_k$ the  set of recognizable series with the \emph{syntactic ideal} $I_{\alpha}$ of codimension $k$ and  refer to $\alpha\in\Rec_k$ as a \emph{recognizable  series of codimension} $k$. Let also  
\begin{equation*}
    \Rec := \bigsqcup_{k\ge 0} \Rec_k, \ \ \Rec_{\le n}:= \bigsqcup_{k\le n} \Rec_k. 
\end{equation*}

Consider the Hilbert scheme $H^{k}=\mathrm{Hilb}^k(\C^2)$ of $k$  points  in $\C^2$ or, equivalently, 
the scheme of  codimension $k$ ideals in $\C[T_1,T_2]$, see~\cite{N}. 

Denote by $\mcT_k$ the tautological bundle over $H^k$  whose  fiber over the point  associated to an ideal  $I$ of codimension $k$ is  the space $A_I=\C[T_1,T_2]/I$. Points of the dual bundle $\mcTcheck_k$ above a point  $I\in H^k$
describe elements of  $A_I^{\ast}=\Hom_{\C}(A_I,\C)$, that  is, linear functionals on $A_I.$ Let 
\begin{equation*}
    \pi  \ : \ \mcTcheck_k \lra H^k
\end{equation*}
be the projection of the bundle onto its base. 
For a point $p\in\mcTcheck_k$ the element $\pi(p)\in H^k$ is the projection of $p$  onto the base of the bundle, and we denote by $I_{\pi(p)}$ the corresponding codimension $k$ ideal of $\C[T_1,T_2]$. 

The point $p$ also defines a linear functional 
$\alpha_p$ on 
\begin{equation*}
    A_{\pi(p)} \ := \ \C[T_1,T_2]/I_{\pi(p)} , \ \ \alpha_p: A_{\pi(p)} \lra \C ,    
\end{equation*}
associated  to $p$. This functional lifts to a functional 
on $\C[T_1,T_2]$, which is  recognizable, contains $I_{\pi(p)}$ in its kernel, and has  codimension at most  $k$. The  latter functional (equivalently, recognizable  power series)  is also denoted  $\alpha_p$.

This functional has the associated ideal $I_p=I_{\alpha_p}\subset \C[T_1,T_2]$ of finite  codimension, the largest ideal in the kernel of functional  $\alpha_p$ on $\C[T_1,T_2]$. There  is an inclusion of ideals 
\begin{equation*}
I_{\pi(p)} \subset I_p. 
\end{equation*}
For  a generic point $p$ on $\mcTcheck_k$ this inclusion  is the equality  $I_{\pi(p)}=I_p$, but for some points $p$ the inclusion is  proper.

Another way to describe the ideal $I_p$ is to consider the  symmetric bilinear form $(~,~)_p$ on $A_{\pi(p)}$ given by 
\begin{equation*}
    (x,y)_p \ := \ \alpha_p(xy),  \ \  x,y\in A_{\pi(p)}.
    \end{equation*}
The  kernel  of the   form $(~,~)_p$ is an ideal $I_p'$ in  $A_{\pi(p)}$ that lifts to  the above ideal $I_p$ in $\C[T_1,T_2]$, and there is an isomorphism $I_p'\cong I_p/I_{\pi(p)}$. The  inclusion $I_{\pi(p)} \subset I_p$ is proper precisely when $I_p'$ is a nonzero ideal, that is, when the bilinear form $(~,~)_p$ is degenerate.  

These ideals are shown in  the diagram below, where the two squares on the left  are pull-backs. The bottom sequence is short exact, and the top row becomes exact upon replacing $I_{\pi(p)}$ by $0$. 
\begin{equation*}
\begin{gathered}
    \xymatrix{
    I_{\pi(p)} \ar@{^{(}->}[r] \ar[d] & I_p \ar[d] \ar@{^{(}->}[r] & \C[T_1, T_2] \ar@{->>}[r] \ar[d] & A_p \ar@{=}[d] \ar[r] & 0 \ar[d]\\
    0 \ar@{^{(}->}[r] & I_p^\prime \ar[r] & A_{\pi(p)} \ar@{->>}[r] & A_p \ar[r] & 0
    }
    \end{gathered}
\end{equation*}

Denote  by $D_k$ the subset of $\mcTcheck_k$ that  consists of  points $p$ such that the inclusion $I_{\pi(p)} \subset I_p $ is proper: 
\begin{equation*}
    D_k \ := \ \{ p\in \mcTcheck_k | I_{\pi(p)} \not= I_p\}.   
\end{equation*}
Recognizable power series $\alpha_p$ for $p\in \mcTcheck_k$ has codimension $k$ (in our  notations, $\alpha_p \in \Rec_k$) precisely when 
$p\in \mcTcheck_k \setminus D_k$. 

If  $p\in D_k$ so that  
\begin{equation*}
\mathrm{codim}(I_p)=m<k = \mathrm{codim}(I_{\pi(p)}), 
\end{equation*}
then recognizable power series  $\alpha_p$ has codimension $m$ strictly less than $k$  and $ \alpha_p \in \Rec_m$. For  example, if $p\in H^k\subset \mcTcheck_k$ is a  point in the zero section  of  $\mcTcheck_k$, so that the linear  map  $\alpha_p$ is identically zero, the ideal $I_p=\C[T_1,T_2]$ has zero codimension and $m=0$. A mild confusion exists in our notations in this case (and  in  this case only), for  then $p=\pi(p)$. 

\vspace{0.2in} 

Going the other way, to a recognizable series $\alpha$  with the  associated ideal $I_{\alpha}$ of codimension $d_{\alpha}=k$ as above we  associate  a point $p_{\alpha}$  of  $\mcTcheck_k$. It is the point in the fiber of  $\mcTcheck_k$ over the ideal $I_{\alpha}$ which describes  functional $\alpha$
on $\C[T_1,T_2]$ and the  induced functional on the quotient algebra  $A_{\alpha}=A_{I_{\alpha}}$. 
  
The  above discussion implies the proposition below. 
\vspace{0.1in}

\begin{prop} \label{prop_assign}
Assigning  $p_{\alpha}$ to $\alpha\in \Rec_k$ and $\alpha_p$ to $p\in \mcTcheck_k \setminus D_k$  establishes a bijection 
\begin{equation*} 
\Rec_k \ \cong \  \mcTcheck_k \setminus D_k.  
\end{equation*} 
\end{prop}
In  particular, $p_{\alpha_p}=p$ and $\alpha_{p_{\alpha}}=\alpha$ for  $p$ and $\alpha$ as in the proposition, so the two assignments are  mutually-inverse bijections. 
\hfill $\square$

Note that the two ideals coincide, $I_{\pi(p)}=I_p$, precisely  when  $\alpha_p$ is a nondegenerate trace map on  $A_{\pi(p)}$. In particular, in this case $A_{\pi(p)}$ is Frobenius. We obtain the following statement.
 
\begin{prop} Points $p\in \mcTcheck_k \setminus D_k$  classify isomorphism classes of data $(A,\epsilon,t_1,t_2)$: a commutative Frobenius  algebra $A$ over $\C$ of dimension $k$ with a non-degenerate trace $\epsilon$ and generators $t_1,t_2$ of $A$.  
\end{prop} 
 
 Not every commutative Frobenius algebra can be generated  by two elements, of course. 
 
 Taking codimension $m\le k$ of $I_p$ into account, one gets the following statement. 
 
 \begin{prop} Associating $\alpha_p$ to $p\in \mcTcheck_k$ gives a  surjective map 
 \begin{equation*}
     \mcTcheck_k \ \lra \ \bigsqcup_{m=0}^k \Rec_m. 
 \end{equation*}
 \end{prop}
 
 Restricting this map  to $D_k$ gives a surjective map 
 \begin{equation*}
     D_k \ \lra \ \bigsqcup_{m=0}^{k-1} \Rec_m,
 \end{equation*}
 while  on the complement to $D_k$ this map is the bijection in Proposition~\ref{prop_assign}. 
 
 \vspace{0.07in}
 
 \emph{Example:} The set $\Rec_0$ is a single point corresponding  to the zero series  $\alpha$, $\alpha_{i,j}=0, i,j\in\Z_+$. The ideal for this  point is the entire algebra $\C[T_1,T_2]$. Points of $\Rec_1$ correspond to hyperplanes (codimension one subspaces) that are ideals $J=(T_1-\lambda_1,T_2-\lambda_2)$ together  with  a nonzero functional on $\C\cong \C[T_1,T_2/J$, determined by its  value $\lambda$ on $1$. Consequently, we can  identify $\Rec_1\cong \C\times \C \times \C^{\times}$ by sending  a point in $\Rec_1$  to the triple $(\lambda_1,\lambda_2,\lambda)$.

\vspace{0.1in} 

{\it  Set-theoretic divisor $D_k$.}
Quasi-projective variety $H^k$ admits an open cover by affine sets $U_{\lambda}$, over all partitions  $\lambda$ of $k, $
see Theorem 18.4 in~\cite[Section 18.4]{MiS}, for example. 
Place partition $\lambda$ in the corner of the first quadrant of the plane so that it covers nodes $(i,j)$ of the square  lattice with $0\le i < \lambda_{j+1}.$ In particular,  it covers $\lambda_1$ nodes on the $x$-axis.

Let $T_{\lambda}$ be the set of monomials $T_1^iT_2^j$ with $(i,j)\in  \lambda$ (in particular, $|T_{\lambda}|=k$) and $T'_{\lambda}$ be  the set of complementary monomials, for pairs $(i,j)\in \Z_+\times \Z_+\setminus \lambda$. Open set $U_{\lambda}\subset H^k$ consists of ideals $I$ such that monomials in $T_{\lambda}$ descend to a basis of $A_I=\C[T_1,T_2]/I$, see~\cite[Section 18.4]{MiS} for details. 

The vector bundle $\mcTcheck_k\lra H^k$ can be trivialized  over $U_\lambda$, being naturally isomorphic to the trivial bundle of functions on the set $T_{\lambda}$. A functional $p$ on $\C[T_1,T_2]/I_{\pi(p)}$ is determined by  its  values on the basis elements  $t\in T_{\lambda}$ of this quotient  space. 

To describe the points $p\in \mcTcheck_k$ with $\pi(p)\in U_{\lambda}$ consider an arbitrary linear functional $\alpha\in (\C T_{\lambda})^{\ast}$, given by its values 
\begin{equation*}
\alpha(T_1^iT_2^j)\in  \C, \ \mathrm{for} \   T_1^iT_2^j\in T_{\lambda},
\end{equation*} 
and an ideal $I\in U_{\lambda}$. Such  pair $(\alpha, I)$ trivializes a pair $(p,\pi(p))$ with  $\pi(p)\in  U_{\lambda}$. For a  pair $u,v\in T_{\lambda}$ take the product $uv$, view it as an element  of $A_I=\C[T_1,T_2]/I$, and then write it as a linear combination of elements in $T_{\lambda}$, allowing to apply $\alpha$  to it explicitly. 

Consider a  matrix $M_{\alpha}$  where  rows  and columns  are labelled by elements of $T_{\lambda}$  and put $\alpha(uv)$ as the entry at the intersection of row $u$ and column $v$. 

\begin{prop} Point $p$  with $\pi(p)\in U_{\lambda}$ is in the subset $D_k$ iff $\det(M_{\alpha})=0$.
\end{prop}
\begin{proof}
   Matrix $M_{\alpha}$ is the {\it Gram} or {\it Hankel matrix} of the bilinear form $(x,y)=\alpha(xy)$ on the associative algebra $A_I$ in the  basis $T_{\lambda}$. A bilinear  form on a finite-dimensional algebra $B$ given by the composition of the multiplication with a fixed linear functional on $B$ is non-degenerate exactly when its  Hankel matrix with respect to some (equivalently, any) basis is non-degenerate,  that is, has a non-zero determinant. 
\end{proof}

Condition that $\det(M_{\alpha})=0$ is locally  a codimension one condition (given by  a single equation), unless the determinant  is identically zero on points $(p,\pi(p))$ with $\pi(p)$ on some irreducible component of  the open subset $U_{\lambda}$ of the Hilbert scheme. To see that the latter case does not happen, observe that a ``generic'' point $I$ on the Hilbert scheme $H^k$ corresponds to a semisimple  quotient (no nilpotent elements in $\C[T_1,T_2]/I$),  with the quotient algebra isomorphic to the product of $k$  fields, 
\begin{equation*}
    \C[T_1,T_2]/I \ \cong \C\times \C\times \dots\times \C. 
\end{equation*}
On this  quotient an open subset  of linear  functionals are non-degenerate, with the  associated bilinear forms having trivial kernels. Indeed, a functional $\alpha$ on the algebra $\prod_{i=1}^k \C$ is non-degenerate iff each of its $k$ coefficients is non-zero.

\vspace{0.1in}

These observations imply the following result. 

\begin{prop}  $D_k$ is a set-theoretic divisor on the variety $\mcTcheck_k$. 
\end{prop}

It is straightforward to check that $D_k$ comes from an  actual divisor on $\mcTcheck_k$. 
For a finite-dimensional $\C$-vector space $V$ define a one-dimensional vector space
\begin{equation*}
\det\, V \ := \  (\Lambda^{\mathrm{top}} V)^\vee=\Lambda^{\mathrm{top}}(V^\vee).
\end{equation*}
The determinant $\det\,\ha$ of a bilinear form $\ha\colon V\otimes V\rightarrow \C$ is an element $\det\,\ha\in (\det\, V)^{\otimes 2}$ defined as the determinant of the matrix of 
$\ha$. Namely,  if $e_1,\ldots ,e_k$ is a basis of $V$ and $e^1,\ldots,e^k$ is the dual basis in $V^\vee$, then $e^1\wedge\cdots\wedge e^k$ is a basis in the one-dimensional space $\det\,V$ and 
\[\det\,\ha := \det || \ha(e_i,e_j)||\; (e^1\wedge\cdots \wedge e^k)^{\otimes 2}.\]

A point $p\in\mcTcheck_k$ defines a symmetric bilinear form $\ha_p(x,y) :=\alpha_p(xy)$ on the fiber $\mcT_{\pi(p)}=I_{\pi(p)}$ of the tautological bundle. The determinant of this form is an element of $(\det\,\mcT_{\pi(p)})^{\otimes 2}$. Hence the pullback line bundle 
\begin{equation*}
\pi^*\left((\det\mcT)^{\otimes 2}\right)\longrightarrow \mcTcheck_k
\end{equation*}
over  $\mcTcheck_k$ has a canonical section $\sigma_{\det}$ given by  $\sigma_{\det}(p) := \det\,\ha_p$. The set $D_k$ is the divisor of zeroes of this section.

\begin{cor}  $D_k$ is the divisor of zeros of the section $\sigma_{\det}$. 
\end{cor}

\vspace{0.1in} 

Each point  of  $\mcTcheck_k\setminus  D_k$ gives  rise to recognizable series $\alpha$  in two variables and  to the  corresponding rigid symmetric monoidal categories, as discussed in the Section~\ref{sec_linear}  and summarized in Section~\ref{subsec_summary}, including category $\tfs_{\alpha}$, the  Deligne  category $\dtfs_{\alpha}$ and  its gligible quotient $\udtfs_{\alpha}$. It  may be interesting to  understand these categories for various $\alpha$'s,  including finding the analogue of the classification result from~\cite{KKO} on when the category $\udtfs_{\,\alpha}$ is semisimple.

 %
 %

\section{Modifications}\label{sec_modif}


\subsection{Adding closed surfaces}\label{subsec_adding} 

$\quad$

\vspace{0.1in} 

Category $\tfs$ can be enlarged to a category $\SS$ with morphisms -- oriented 2D cobordisms  (surfaces)  with corners between oriented 1D manifolds  with corners. Extensions of 2D TQFTs to  this category have been widely studied~\cite{MS,LP,C,SP}. An oriented 1D manifold  with corners is diffeomorphic to a  disjoint union of finitely-many oriented intervals and  circles. We adopt a minimalist approach and  choose one manifold for each such  diffeomorphism class. Consequently, objects of $\SS$ are pairs $\undern = (n_1,n_2)$ of non-negative integers, and an object $\undern$ is represented  by a  fixed disjoint  union $W(\undern)=W(n_1,n_2)$ of $n_1$ intervals and $n_2$ circles. Morphisms from $\undern$  to  $\underm=(m_1,m_2)$ are compact oriented 2D cobordisms $M$, possibly with corners, with both \emph{horizontal} and  \emph{side} boundary and  corners  where these two different boundary types  meet: 
\begin{equation*}
    \partial M = \partial_h M \cup \partial_v M, \ \partial_h M =W(m_1,m_2) \sqcup (-W(n_1,n_2)).  
\end{equation*}
Cobordisms that  are diffeomorphic  rel boundary define  the same  morphisms. Category $\SS$ contains  $\tfs$ as a subcategory. 

$\SS$ is a rigid symmetric monoidal category, with self-dual objects. The unit  object $\one$ is the empty one-manifold  $W(0,0)$. Its endomorphism monoid is freely generated  by diffeomorphism types of compact connected surfaces with boundary. The latter are classified by surfaces $S_{\ell,g}$ with $\ell$ boundary components and  of  genus $g$, one  for each pair $(\ell,g)$, $\ell,g\in \Z_+$. The difference with endomorphisms of the  unit object of $\tfs$ is that in $\SS$ closed surfaces are  allowed, which corresponds to $\ell=0$ and  surfaces  $S_{0,g}$, over all $g\in \Z_+$.

Multiplicative evaluations $\beta$ of endomorphisms of the  unit object are  again encoded by a power series
\begin{equation}\label{eq_Z_rat_4}
    \tZ_{\beta}(T_1,T_2) = \sum_{k,g\ge 0}\beta_{k,g} T_1^k T_2^g, \ \ 
     \beta=(\beta_{k,g})_{k,g\in \Z_+}, \ \ \alpha_{k,g} \in \kk , 
\end{equation} 
with the first index shifted  by $1$ compared to  evaluations for $\tfs$. We  changed the label from $\alpha$ in evaluations in $\tfs$ to $\beta$ in $\SS$  to make it  easier to compare evaluations in these  two categories. Now the  coefficient  
\begin{equation*}
\beta_{k,g} = \beta(S_{k,g})
\end{equation*} 
is the  evaluation of connected genus $g$  surface with $k$ boundary components  rather than with $k+1$ components as in the  $\tfs$ case, see earlier. 

To relate these two  power series encodings, in formulas (\ref{eq_Z_rat_2}) and  (\ref{eq_Z_rat_1}) versus  (\ref{eq_Z_rat_4}), start with $Z_{\alpha}(T_1,T_2)$ as in (\ref{eq_Z_rat_1}) and also form a one-variable power  series 
\begin{equation*}
    Z_{\gamma}(T_2)=\sum_{k\ge 0} \gamma_k T_2^k, \ \ \gamma_k \in \kk.  
\end{equation*}
To the pair $(Z_{\alpha},Z_{\gamma})$ assign the series 
\begin{equation} \label{eq_alpha_bg}
 \tZ_{\beta}(T_1,T_2) = T_1 Z_{\alpha}(T_1,T_2) + Z_{\gamma}(T_2). 
\end{equation}
Adding coefficients of $Z_{\gamma}$ to the data provided by  $Z_{\alpha}$ precisely means  that we now include evaluations  of  closed surfaces, via coefficients $\gamma_k$ (for a closed surface genus $k$). The scaling factor $T_1$ in the formula is needed to match the discrepancy in the evaluation conventions in the two  categories $\tfs$ and  $\SS$. Formula (\ref{eq_alpha_bg}) gives a bijection between series encoded by $\beta$  and those encoded by $(\alpha,\gamma)$. Starting from $\tZ_{\beta}$, one recovers $Z_{\alpha}$ and $Z_{\gamma}$ as 
\begin{eqnarray*}
 Z_{\gamma}(T_2) & = & \tZ_{\beta}(0,T_2)  \\
 Z_{\alpha}(T_1,T_2)  & = &  (\tZ_{\beta}(T_1,T_2)-\tZ_{\beta}(0,T_2))/T_1. 
\end{eqnarray*}

From $\SS$ pass to its $\kk$-linearization $\kk\SS$ by allowing finite $\kk$-linear combinations of morphisms in $\SS$. 
Given series $\beta$, we can define analogues of categories $\vtfs_{\alpha}$  and    $\tfs_{\alpha}$ in (\ref{eq_seq_cd}). Denote these new  categories by $\vSS_{\beta}$ and $\SS_{\beta}$:
\begin{itemize}
    \item In $\vSS_{\beta}$ one evaluates floating  components to  elements of $\kk$ via $\beta$. A connected component is  \emph{floating} if it  has no  horizontal boundary. 
    \item To form category $\SS_{\beta}$ we quotient category $\vSS_{\beta}$  (alternatively, category $\kk\SS$) by the  two-sided ideal of \emph{negligible morphisms}, defined in the same way as for $\tfs$.
\end{itemize}

We say that evaluation $\beta$ (or series  $\tZ_{\beta}$) is \emph{recognizable} if category $\SS_{\beta}$ have finite-dimensional hom spaces.

\begin{prop} \label{prop_recog_beta}
$\beta$ is recognizable iff  the power series $\tZ_{\beta}$ has the form
\begin{equation}\label{eq_recog_10}
    \tZ_{\beta}(T_1,T_2) = \frac{\wtP(T_1,T_2)}{\wtQ_1(T_1)\wtQ_2(T_2)},
\end{equation}
where $\wtQ_1(T_1),\wtQ_2(T_2)$ are one-variable polynomials and $\wtP(T_1,T_2)$ is a two-variable polynomial, all with coefficients in the field $\kk$.  
\end{prop} 

It is easy to see  that series $\beta$ is recognizable iff hom  spaces 
\begin{equation*}
\Hom(\one,(1,0)) \ \  \mathrm{and}  \ \ \Hom(\one,(0,1))
\end{equation*}
in $\SS_{\beta}$  are finite-dimensional. These are the hom spaces from the empty  1-manifold $W(0,0)$ (representing the unit object $\one$)  to an  interval $W(1,0)$ and a circle $W(0,1)$, respectively. Necessity of this condition is obvious. Vice versa, if these homs are finite-dimensional, by analogy  with the proof  of  Proposition~\ref{prop_recog_alpha}, there are skein relations allowing to reduce some large number of handles (respectively, holes) on a connected component  to  a  linear  combination  of otherwise identical cobordisms but with fewer handles (respectively holes).  The  rest of the proof of Proposition~\ref{prop_recog_beta} follows  that of Proposition~\ref{prop_recog_alpha}. 

$\square$

\begin{cor}
   Series $\beta$  is recognizable iff the corresponding  series  $\alpha$ and  $\gamma$  are both recognizable. 
\end{cor}

Note that, when $\alpha$ and $\gamma$ are   recognizable, their rational function  presentation  may have very different  denominators, 
\begin{equation*}
    Z_{\alpha}(T_1,T_2)=\frac{P(T_1,T_2)}{Q_1(T_1)Q_2(T_2)}, \  \  Z_{\gamma}(T_2) = \frac{P_{\gamma}(T_2)}{Q_{\gamma}(T_2)},
\end{equation*}
so that 
\begin{eqnarray*}
    \tZ_{\beta}(T_1,T_2) & = & \frac{T_1 P(T_1,T_2)}{Q_1(T_1)Q_2(T_2)} + \frac{P_{\gamma}(T_2)}{Q_{\gamma}(T_2)} \\
    &  = & \frac{T_1 P(T_1,T_2)Q_{\gamma}(T_2)+Q_1(T_1)Q_2(T_2)P_{\gamma}(T_2)}{Q_1(T_1)Q_2(T_2)Q_{\gamma}(T_2)}.
\end{eqnarray*}
For generic polynomials, there are  no cancellations and  \begin{equation*}
    \wtQ_1(T_1) = Q_1(T_1), \ \ \wtQ_2(T_2) = Q_2(T_2)Q_{\gamma}(T_2)
\end{equation*}
are  the denominators in the minimal rational presentation (\ref{eq_recog_10}) for $\tZ_{\beta}$. 

\vspace{0.1in} 

For recognizable $\beta$, the state spaces 
\begin{equation*}
A_{\beta}(1,0) := \Hom_{\SS_{\beta}}(\one,(1,0)), \  \ 
A_{\beta}(0,1) := \Hom_{\SS_{\beta}}(\one,(0,1)),
\end{equation*} 
of homs from the unit object $\one=(0,0)$ to the interval and the circle objects, respectively, are both commutative Frobenius algebras. Annuli, viewed as morphisms between $(1,0)$ and $(0,1)$, see Figure~\ref{fig5_1}, give linear maps
\begin{equation*}
\delta_0  \ : \ A_{\beta}(1,0) \lra A_{\beta}(0,1), \ \ \ \ 
\delta_1  \ : \ A_{\beta}(0,1) \lra A_{\beta}(1,0)
\end{equation*} 
between  the underlying vector  spaces.  

\begin{figure}[h]
\psfrag{d0}{$\delta_0$}
\psfrag{d1}{$\delta_1$}
\begin{center}
\includegraphics[scale=0.75]{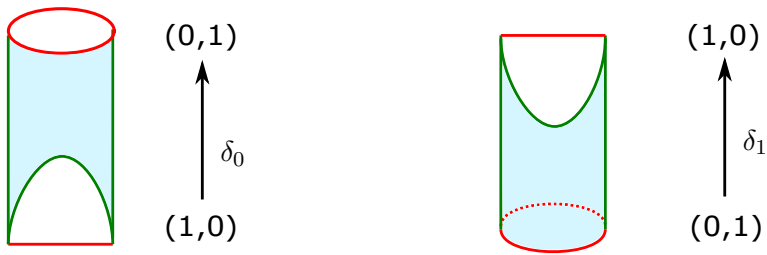}
\caption{Maps $\delta_0$ and $\delta_1$. }
\label{fig5_1}
\end{center}
\end{figure}

Consider the \emph{hole} and \emph{handle} endomorphisms $b_1,b_2$ of the interval and $c_1,c_2$ of the circle, respectively, in Figure~\ref{fig5_2} top. 

\begin{figure}[h]
\psfrag{b1}{$b_1$}
\psfrag{b2}{$b_2$}
\psfrag{c1}{$c_1$}
\psfrag{c2}{$c_2$}
\psfrag{ub1}{$\underline{b}_1$}
\psfrag{ub2}{$\underline{b}_2$}
\psfrag{uc1}{$\underline{c}_1$}
\psfrag{uc2}{$\underline{c}_2$}
\begin{center}
\includegraphics[scale=0.65]{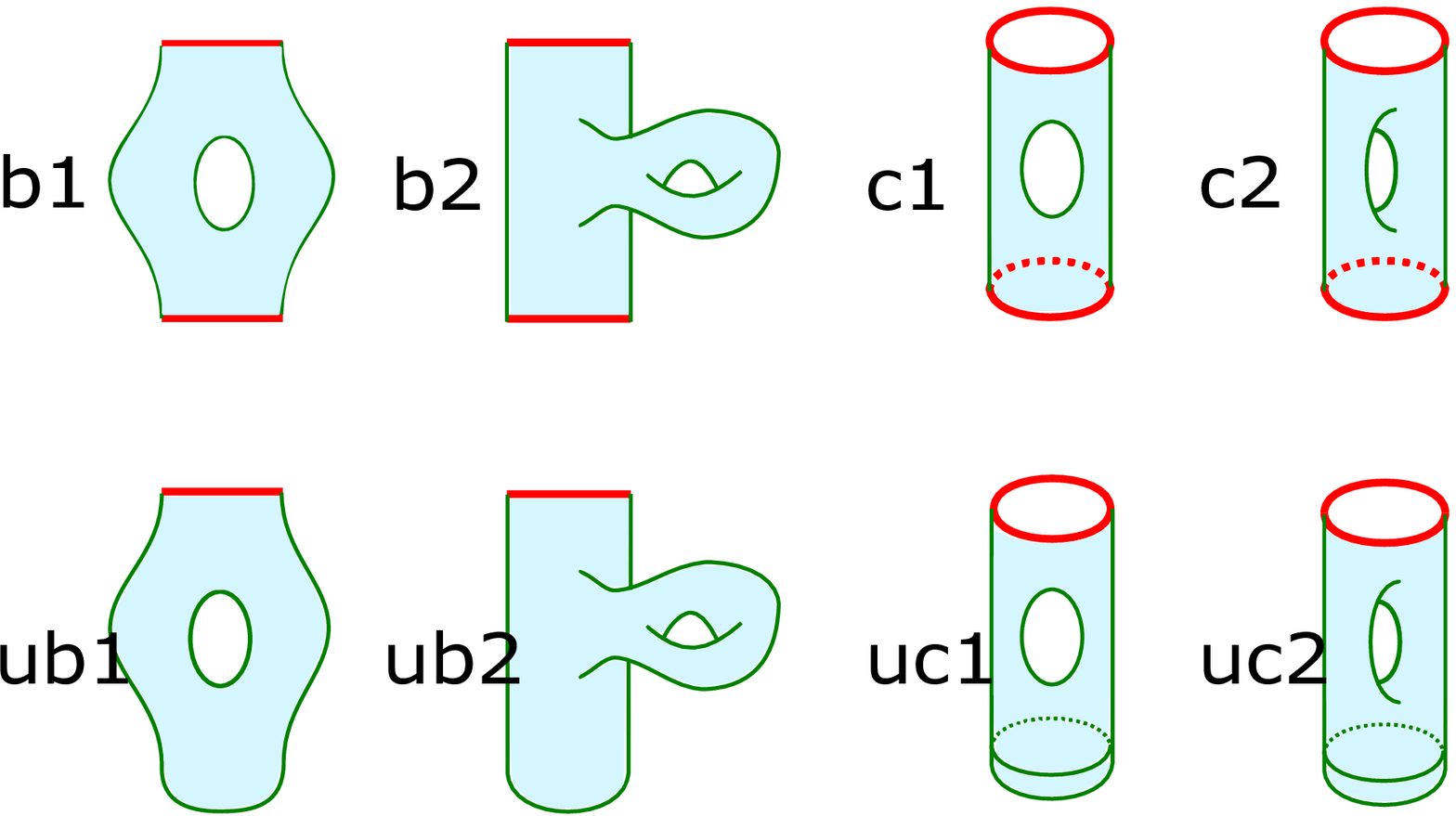}
\caption{Endomorphisms  $b_1,b_2$ of the interval, endomorphisms $c_1,c_2$ of the circle and corrresponding elements $\undb_1,\undb_2$ of $A_{\beta}(1,0)=\Hom(\one,(1,0))$ and elements $\undc_1,\undc_2 \in  
A_{\beta}(0,1)=\Hom(\one,(0,1))$.}
\label{fig5_2}
\end{center}
\end{figure}

Multiplications in algebras $A_{\beta}(1,0)$ and $A_{\beta}(0,1)$ are given by pants and flat pants cobordisms, see Figure~\ref{fig5_3}, where the cobordisms for the unit and trace morphisms on 
$A_{\beta}(1,0)$ and $A_{\beta}(0,1)$ are shown  as well. 

\begin{figure}[h]
\psfrag{M}{$m$}
\psfrag{i}{$\iota$}
\psfrag{E}{$\epsilon$}
\psfrag{M1}{$m^\prime$}
\psfrag{i1}{$\iota^\prime$}
\psfrag{E1}{$\epsilon^\prime$}
\begin{center}
\includegraphics[scale=0.65]{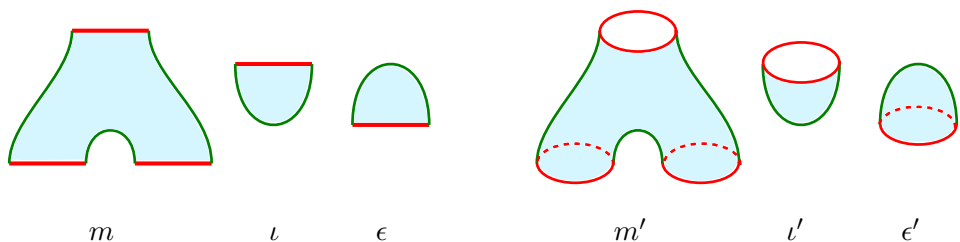}
\caption{Flat pants and pants cobordisms, together with the other structure maps $\iota,\epsilon$ and $\iota',\epsilon'$  (units $\iota,\iota'$ and counits $\epsilon,\epsilon'$) of commutative  Frobenius algebras $A_{\beta}(1,0)$ and $A_{\beta}(0,1)$. }
\label{fig5_3}
\end{center}
\end{figure}

Take endomorphisms $b_1,b_2,c_1,c_2$ of the interval and circle and cap them off at the  bottom with the unit morphisms $\iota$ and $\iota'$ for the interval and circle (see Figure~\ref{fig5_3}) to get elements $\undb_1=b_1\iota,\undb_2=b_2\iota$ in $A_{\beta}(1,0)$ and  elements $\undc_1=c_1\iota', \undc_2 = c_2\iota'$ in $A_{\beta}(0,1)$, shown in Figure~\ref{fig5_2}.

The analogue of Proposition~\ref{prop_y} holds in $\SS$, and the  ``interval'' Frobenius algebra $A_{\beta}(1,0)$ is generated by commuting elements  $\undb_1,\undb_2$ (the hole and handle elements). Likewise,  the  ``circle'' Frobenius algebra $A_{\beta}(0,1)$ is generated by commuting  hole and  handle  elements $\undc_1$ and $\undc_2$. 

Endomorphisms $b_1,b_3$ of the interval in the category $\SS$ are different  (endomorphism $b_3$ is also shown in Figure~\ref{fig2_2}), but they induce the same map  on $A_{\beta}(1,0)$, see Figure~\ref{fig5_4}. There, $x\in A_{\beta}(1,0)$  can be written as a linear combination of monomials $\undb_1^n\undb_2^m$, with $b_3$ acting  by 
\begin{equation*}
b_3 \undb_1^n\undb_2^m=\undb_1^{n+1}\undb_2^m = b_1\undb_1^n\undb_2^m. 
\end{equation*}

\begin{figure}[h]
\psfrag{x}{$x$}
\psfrag{b3}{$b_3$}
\psfrag{b1}{$b_1$}
\begin{center}
\includegraphics[scale=0.75]{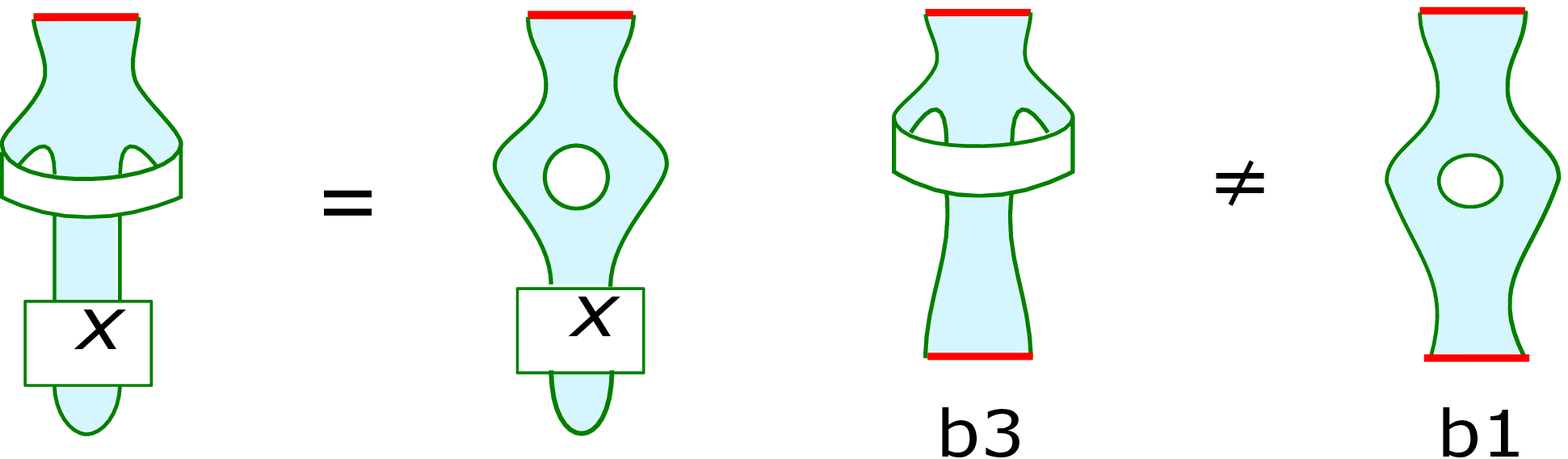}
\caption{$b_3x=b_1x$ for any $x\in A_{\beta}(1,0)$. $b_3\not= b_1$ as $\End((1,0))$ in $\SS$ (and in $\SS_{\beta}$, in general).}
\label{fig5_4}
\end{center}
\end{figure}

Trace maps 
\begin{equation} \label{eq_traces}
\epsilon \ : \  A_{\beta}(0,1)\lra \kk, \ \ \ \ 
\epsilon' \ : \  A_{\beta}(1,0)\lra \kk,
\end{equation} 
given by capping off the interval with a disk, respectively the circle with a cap, turn these two commutative algebras into Frobenius algebras (for recognizable $\beta$).

Compositions of $\delta_0$ and $\delta_1$ are endomorphisms of the  interval and the circle in $\SS$ (and in  $\SS_{\beta}$) and satisfy
\begin{eqnarray*}
    \delta_1\delta_0 & = &  b_3, \ \ \ \ \,
    \delta_0\delta_1 \  = \   c_1, \\
    \delta_1 c_1 & = & b_1 \delta_1, \ \ 
     \delta_1 c_2 \  = \  b_2 \delta_1, \\
     \delta_0 b_1  & = & c_1 \delta_0, \ \ 
    \delta_0 b_2  \ =  \ c_2 \delta_0.
\end{eqnarray*}

In particular, maps $\delta_0,\delta_1$ intertwine the hole endomorphisms $b_1,c_1$ of the interval and the  circle. They  also intertwine the handle endomorphisms $b_2,c_2$ of the  interval and  the  circle.  

Their two compositions produce the hole endomorphisms of the interval and the circle. 

The map $\delta_1$ is a surjective unital homomorphism of commutative algebras, while the map $\delta_0$ is an  injective homomorphism of cocommutative coalgebras, with comultiplications given by the dual of the multiplications on  these Frobenius algebras. In particular, $\delta_0$ respects  traces, in  the sense  that $\epsilon'\delta_0 = \epsilon.$

\vspace{0.1in}
A 
recognizable power series $\beta$ is encoded by a commutative  Frobenius algebra (the state space of  a circle) $A_{\beta}(0,1)$ with generators $\undc_1,\undc_2$ and non-degenerate trace map $\epsilon'$ such that 
\begin{equation} \label{eq_beta_epsilon}
    \beta_{\ell,g} = \epsilon'(\undc_1^{\ell}\undc_2^g), \ \  \ell,g\in \Z_+.
\end{equation}
Further unwrapping this data, to a 
recognizable power series $\beta$ we can associate
\begin{itemize}
\item 
Two commutative Frobenius algebras $A(1,0)=A_{\beta}(1,0)$ and $A(0,1)=A_{\beta}(0,1)$ with generators $\undb_1,\undb_2$ and $\undc_1,\undc_2$, respectively (hole and handle elements). 
\item
Non-degenerate traces $\epsilon$ and $\epsilon'$ as in (\ref{eq_traces}), subject to (\ref{eq_beta_epsilon})  and 
\begin{equation*} 
    \beta_{\ell+1,g} = \epsilon(\undb_1^{\ell}\undb_2^g), \ \  \ell,g\in \Z_+.
\end{equation*}
\item Linear maps $\delta_0,\delta_1$:
\begin{center}
\begin{tikzcd}
A_{\beta}(1,0)\arrow[r,shift left, "\delta_0" ]
    &
    A_{\beta}(0,1) \arrow[l,shift left,"\delta_1"]
\end{tikzcd}
\end{center} 
that  intertwine the action of handle elements $\undb_2$ and $\undc_2$. The  hole elements are given by   
\begin{equation*}
\undb_1=\delta_1\delta_0(1), \ \  \ 
\undc_1=\delta_0  \delta_1(1).  
\end{equation*} 
\item $\delta_1$  is a surjective unital homomorphism  of commutative algebras. 
\end{itemize}

\vspace{0.1in} 

The reader may want  to constrast the data coming from a recognizable series $\beta$  as above, with both algebras $A_{\beta}(0,1)$ and  $A_{\beta}(1,0)$ commutative Frobenius, with that given by a 2-dimensional TQFT with corners~\cite{MS,LP,C,SP} where the Frobenius algebra $B$ associated  to the interval  is not necessarily commutative and the algebra associated to the circle is  related to the center of  $B$. 

\vspace{0.1in} 

To a  recognizable series $\beta$ there is associated 
a  finite codimension ideal $I_{\beta}\subset \kk[T_1,T_2]$ describing relations on the hole and handle endomorphisms along any component of a surface. Starting with the \emph{viewable} category $\vSS_{\beta}$, described earlier, where floating components are evaluated via $\beta$, we impose relations in $I_{\beta}$ on hole and handle endomorphisms  along any component. The resulting category, denoted $\sSS_{\beta}$ (the \emph{skein} category) has finite-dimensional hom spaces. 

From the skein category we can  pass to the already defined gligible quotient $\SS_{\beta}$ by  taking the quotient of  $\sSS_{\beta}$ by the  ideal of  negligible  morphisms. This  ideal comes from the trace on $\sSS_{\beta}$ or, equivalently, from  the bilinear form given by pairing  of cobordisms. 

Taking the additive Karoubi closure of $\sSS_{\beta}$ results in the Deligne category $\dSS_{\beta}$. 

Taking the  quotient of $\dSS_{\beta}$ by the ideal of negligible morphisms produces the category  $\udSS_{\beta}$. Alternatively, this category is  equivalent to the  additive  Karoubi closure of $\SS_{\beta}$, and the square below  is commutative. 

The  following diagram summarizes these categories and functors (compare with (\ref{eq_seq_cd}), (\ref{eq_words}), and~\cite{KS}).  

\begin{equation} \label{eq_seq_cd_5}
\begin{CD}
\SS  @>>> \kk\SS @>>> \vSS_{\beta} @>>> \sSS_{\beta} @>>>  \dSS_{\beta} \\
@.   @. @.   @VVV   @VVV  \\    
 @.  @.   @.  \SS_{\beta}  @>>>  \udSS_{\beta}
\end{CD}
\end{equation}

Each  of the  four  categories in the vertices of the commutative square has finite-dimensional hom spaces between its objects. 


\subsection{Coloring side boundaries of cobordisms}  

$\quad$

\vspace{0.1in} 

Fix a  natural  number $r\ge 1$ and  consider a modification $\tfsr$ of  the category $\tfs$ where side  boundaries of cobordisms are colored by numbers from $1$ to $r$.  Let $\N_r= \{1,\dots,  r\}$ be the set of colors. A morphisms in  $\tfsr$ is  a tf-surface $x$, up to rel boundary diffeomorphisms, such that  each  side (or vertical) boundary component of $x$ carries a label  from  $\N_r$. Coloring of  $x$ induces a coloring on the set of  corners of  $x$, that  is, on endpoints of  the one-manifold $\partial_h x$ which is the horizontal boundary of $x$, see Figure~\ref{fig6_1}.

\begin{figure}[h]
\begin{center}
\includegraphics[scale=0.75]{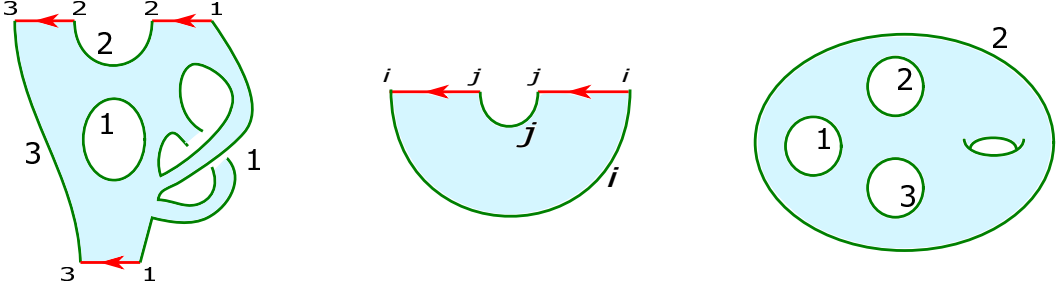}
\caption{Left: A morphism in $\tfsr$ from the colored interval $(3,1)$ to the union $(3,2)\sqcup (2,1)$ of two colored  intervals. Middle: the dual of object $(i,j)$ is  the  object   $(j,i)$. Right: a connected  floating component of genus $1$ and the sequence $(1,2,1)$. It  has one  boundary circle of colors $1$ and $3$ each and two circles of color $2$.}
\label{fig6_1}
\end{center}
\end{figure}
Consequently, each boundary  interval $I$ of $x$, being oriented, gets an induced ordered sequence $(r_1(I),r_0(I))$ of two colors. We consider a skeletal version of $\tfsr$, choosing only one object for  each isomorphism  class. An object $a$ then  is determined  by  the  $r\times r$ matrix $M=M(a)$ with the $(i,j)$-entry the number  of intervals in $a$  colored $(i,j)$.    

Thus, objects $a$ are described by  $r\times r$ matrices of non-negative integers counting number  of colored intervals in $a$. We can call these objects \emph{$r$-colored} or \emph{$r$-labelled} thin one-manifolds or \emph{$r$-boundary colored} thin one-manifolds. An object  can also be described by  a list of  colored intervals in it.

This  skeletal version is still  rigid tensor, with the obvious tensor product.
The unit object $\one = \emptyset$ corresponds to  the matrix of size $0\times 0$.

The notion of a connected component, floating and viewable components of a morphism are  defined as for  $\tfs$. 
Commutative  monoid $\End(\emptyset)$ of endomorphisms  of the empty  one-manifold $\emptyset$ is a  free  abelian monoid  generated  by diffeomorphism classes of  connected floating  $r$-colored tf-surfaces. Such a surface $S$ is classified by its genus $g\ge 0$ and a sequence of  $r$ non-negative  integers $\undn=(n_1,\dots,n_r)$, where  $n_i$ is the  number  of boundary components of  color $i$.
Denote such component by $S_{\undn,g}$. Figure~\ref{fig6_1} right shows the component  $S_{(1,2,1),1}$. 

For each color $i\le r$  there is an embedding of  $\tfs$ into $\tfsr$ by coloring each side boundary of morphisms in $\tfs$ by $i$. Each horizontal interval is then an $(i,i)$-interval. 

For a  morphisms between  two objects in $\tfsr$ to exist, there must exist  a suitable matching between the colorings of their endpoints. For instance, there are no morphisms from the empty object $\emptyset$   to $(i,j)$ interval if $i\not=j$, since the $i$ and $j$ endpoints must belong to  the same side interval and have the same coloring. There are morphisms from $\emptyset$ to $(i,j)\sqcup (j,i)$ but no morphisms from $\emptyset$  to $(i,j)\sqcup  (i,j)$  for $i\not= j$, since matching the two $i$'s via a side interval is not possible with our orientation setup. 

As usual,  denote  by $\kk\tfsr$ the $\kk$-linear version of $\tfsr$, with the same  objects as $\tfsr$ and morphisms --  $kk$-linear combinations of morphisms in $\tfsr.$

\vspace{0.1in} 

The  construction of  evaluation categories and  recognizable (or rational) series can be extended from $\tfs$ to  $\tfsr$ in a direct   way. 

An  evaluation $\alpha$  is a multiplicative homomorphism  from the  monoid $\End(\emptyset)$ of floating colored tf-surfaces to a field $\kk.$ Such  an  evaluation  is determined  by its  values on connected floating surfaces $S_{\undn,g}$. Let 
\begin{equation}
    Z_{\alpha}(T_0,\dots, T_{r})=\sum_{\undn,g} \alpha_{\undn,g}\, T_0^g T^{\undn}, \ \ \alpha_{\undn,g}\in \kk
\end{equation}
be a formal power series  in $r+1$ variables, with 
\begin{equation*}
    T^{\undn}\ := \ T_1^{n_1}\dots T_r^{n_r}, \ \ 
    \undn=(n_1,\dots, n_r), \ \ n_i\in \Z_+
\end{equation*}
where  $T^{\undn}$ is a monomial in $T_1,\dots, T_r$. Thus, $T_0$ is the \emph{genus} variable and $T_1,\dots, T_r$ are \emph{color} variables. Coefficient $\alpha_{\undn,g}$ at $T_0^gT_1^{n_1}\dots T_r^{n_r}$ encodes the evaluation of floating connected surface $S_{\undn,g}$. 

Since each component of a tf-surface has non-empty boundary, coefficients at $T_0^g$, with  $\undn=\mathbf{0}=(0,\dots, 0)$ do not  appear in this formal sum.  We  set them to zero and  extend the sum to these indices by  setting
\begin{equation}\label{eq_is_zero}
    \alpha_{\mathbf{0},g}  =0,\ \ g\in \Z_+.
\end{equation}
Thus, our  power series has the property that 
\begin{equation}\label{eq_Z_satisfy}
    Z_{\alpha}(T_0,0,\dots, 0)= 0 . 
\end{equation}
We  can also view $\alpha$ as a linear map of vector spaces
\begin{equation*}
    \alpha \ :  \ \kk[T_0,\dots, T_r]\lra \kk
\end{equation*}
subject to condition (\ref{eq_is_zero}), that is, $\alpha(T_0^g)=0$,  $g\ge 0$.  

\vspace{0.1in} 

To $\alpha$ we assign the category
$\vtfsr_{\alpha}$,  the quotient of $\kk\tfsr$  by the relations that a connected floating component diffeomorphic to  $S_{\undn,g}$ evaluates  to $\alpha_{\undn,g}$. This is the category of \emph{viewable} $r$-colored tf-surfaces with the $\alpha$-evaluation. 

\vspace{0.1in} 

Categor $\vtfsr_{\alpha}$ carries a natural trace form given on an endomorphism $x$ of an  object $a$ by closing $x$  into a floating surface $\widehat{x}$  and evaluating this  surface via  $\alpha$,  see  Figure~\ref{fig3_3}, where  now  side  boundaries are $r$-colored. If $x$ is not a single cobordism but a linear combination, we use linearity of  the trace to define $\tral(x)=\alpha(\widehat{x})$.

Denote by  $J_{\alpha}$ the two-sided ideal of negligible morphisms in $\vtfsr_{\alpha}$ for this trace map. Define the \emph{gligible} cobordism category $\tfsr_{\alpha}$ as the  quotient of $\vtfsr_{\alpha}$  by the ideal $J_{\alpha}$: 
\begin{equation*}
    \tfsr_{\alpha}  \ := \ \vtfsr_{\alpha}/J_{\alpha}.
\end{equation*}

We say that evaluation $\alpha$ is \emph{rational} or \emph{recognizable} if category $\tfsr_{\alpha}$ has  finite-dimensional hom spaces. 

\begin{prop} \label{prop_3_prop}
The following properties  are equivalent.  
\begin{enumerate}
    \item $\alpha$ is  recognizable.
    \item Hom spaces $\Hom(\emptyset,(i,i))$  from the empty one-manifold to the $(i,i)$-interval are finite-dimensional in $\tfsr_{\alpha}$ for all $i=1,\dots, r$.
    \item Power series $Z_{\alpha}$ has the form
    \begin{equation*}
        Z_{\alpha}(T_0,\dots, T_r) = \frac{P(T_0,\dots, T_r)}{Q_0(T_0)Q_1(T_1)\dots Q_r(T_r)},
    \end{equation*}
    where $P$ is a polynomial in $r+1$ variables and  $Q_0,\dots, Q_r$  are one-variable polynomials, with $Q_i(0)\not= 0$, $i=0,\dots,r$. 
\end{enumerate}
\end{prop} 
Polynomials $Q_i$ can be normalized so that $Q_i(0)=1$ for all $i$. Power series $Z_{\alpha}$ also satisfies equation (\ref{eq_Z_satisfy}). 

\begin{proof}
The proof is essentially  the same as in $r=1$ case, when all side components carry the same color and  there is no  need to mention colors.  Proof of Proposition~\ref{prop_recog_alpha} carries directly to the case of arbitrary $r$. 
\end{proof}

Take any floating component $S$ and a monomial $T=T_0^g T_1^{n_1}\dots T_r^{n_r}$. Define $S(T)$ as the surface $S$ with additional $g$ handles and additional $n_i$ holes with boundary colored $i$,  for  $i=1,\dots, r.$

Given a linear combination $y  = \sum \mu_i T_i$ of monomials, define $S(y)=\sum \mu_i S(T_i)$ as the linear combination of corresponding floating surfaces. Evaluation $\alpha(S(y))$ is an element of the ground field $\kk$. 

\vspace{0.1in} 

Given $\alpha$, we can then define the syntactic ideal $I_{\alpha}\subset \kk[T_0,\dots, T_r]$. 
Namely, $y\in I_{\alpha}$ if $\alpha(S(y))=0$ for any floating $S$.

\begin{prop} $\alpha$ is recognizable iff the ideal $I_{\alpha}$ has finite codimension in $\kk[T_0,\dots, T_r]$. 
\end{prop} 

Thus, for recognizable $\alpha$, one can define the \emph{skein} category $\stfsr_{\alpha}$ as the quotient of $\vtfsr_{\,\alpha}$ by the relations that inserting any $y\in I_{\alpha}$ into a cobordism is zero. 
Category  $\stfsr_{\alpha}$ has finite-dimensional hom spaces. It also has  the  ideal of negligible  morphisms, with the quotient  category isomorphic to $\tfsr_{\,\alpha}$. One can then  define the analogue of the Deligne category for $\stfsr_{\alpha}$ by taking its additive Karoubi closure and define the glibigle quotient of the latter. The resulting diagram of categories and functors  below  mirrors diagrams (\ref{eq_seq_cd}),  (\ref{eq_seq_cd_5}), and the  corresponding diagram in~\cite{KS}. The square is commutative. 

\begin{equation} \label{eq_seq_cd_11}
\begin{CD}
\tfsr  @>>> \kk\tfsr @>>> \vtfsr_{\,\alpha} @>>> \stfsr_{\,\alpha} @>>>  \dtfsr_{\,\alpha} \\
@.   @. @.   @VVV   @VVV  \\    
 @.  @.   @.  \tfsr_{\,\alpha}  @>>>  \udtfsr
\end{CD}
\end{equation}

\vspace{0.1in} 

Condition (\ref{eq_Z_satisfy})  on the  power series $Z_{\alpha}$ seems rather unnatural. It can be  removed by passing to the larger category, as in Section~\ref{subsec_adding}, where now  closed components are allowed. Objects of the new category that extends $\tfsr$ are disjoint  unions of oriented intervals (with endpoints colored by elements of $\N_r$) and circles.  Morphisms are two-dimensional oriented cobordisms between  these collections, with side boundary intervals and side circles colored by elements of  $\N_r$. In  the definition of evaluation $\alpha$ we can now omit condition   (\ref{eq_is_zero}) or, equivalently, restriction (\ref{eq_Z_satisfy}) on the power series $Z_{\alpha}$.

\vspace{0.1in} 

Definition and basic properties or recognizable series now  work as in the $\tfsr$ case. In the analogue of Proposition~\ref{prop_3_prop} for this modification, property (2) is  replaced by the condition that  the state  space of the circle is finite-dimensional (hom space $\Hom(\emptyset,\mathbb{S}^1)$ is finite-dimensional). This is  due to the surjection from the state space of the circle to that of the  interval $(i,i)$ induced by the map $\delta_1$  in Figure~\ref{fig5_2} with the side (vertical) interval colored $i$.  
It is straightforward to set up the analogue of the diagram (\ref{eq_seq_cd_11}) of categories and functors  for this case as well, for recognizable $\alpha$.




\end{document}